\documentclass[10pt, a4paper]{article}
\usepackage[margin=1in]{geometry}

\usepackage{setspace}
\doublespace   
\usepackage{natbib}
\usepackage{filecontents}
\usepackage[toc,page]{appendix}

\usepackage{filecontents}

\usepackage{verbatim,amscd,epsfig,amsmath,amsthm,amstext,amssymb,bbm,hyperref,appendix}
\usepackage{epstopdf}
\usepackage[francais,english]{babel}
\usepackage{bbm}

\providecommand\mathbb{\bf}
\newcommand\R{{\mathbb R}}

\newcommand{\E}{\mathbb{E}}

\DeclareMathOperator*{\argmax}{argmax} 

\usepackage{adjustbox}
\usepackage{layouts}
\newcounter{theoremcounter}

\newtheorem{lemma}[theoremcounter]{Lemma}

\newtheorem{definition}[theoremcounter]{Definition}

\newtheorem{problem}[theoremcounter]{Problem}
\newtheorem{assumption}[theoremcounter]{Assumption}

\newtheorem{game}[theoremcounter]{Game}
\newtheorem{mechanism}[theoremcounter]{Mechanism}
\newenvironment{remark}
      {\par\medbreak\noindent\textbf{Remark.~}}
      {\hbox{ }\medbreak}
\newcounter{steps}

\usepackage{color,xcolor}
\usepackage{float}
\usepackage{mathdots}
\usepackage{upgreek}

\usepackage{mathrsfs}
\usepackage{setspace}
\usepackage{relsize}
\usepackage{caption}

\usepackage[font=small,labelfont=bf,width=0.88\textwidth]{caption}

\usepackage{subcaption}
\usepackage{enumerate}
\usepackage{rotating}
\usepackage{natbib}
\usepackage{graphicx}
\usepackage{enumitem}   
\usepackage{acro}
\newcommand{\n}{\noindent}

\newcommand{\vb}{\mbox{\boldmath{$b$}}}

\newcommand{\vx}{\mbox{\boldmath{$x$}}}

\newcommand{\vO}{\mbox{\boldmath{$O$}}}

\newcommand{\vX}{\mbox{\boldmath{$X$}}}

\newcommand{\valpha}{\mbox{\boldmath{$\alpha$}}}

\newcommand{\vpi}{\mbox{\boldmath{$\pi$}}}

\newcommand{\1}{\mbox{\boldmath{$\mathbbm{1}$}}}
\newcommand*\diff{\mathop{}\!\mathrm{d}}
\newcommand{\Prob}{\ensuremath{\mathbb{P}}}

\usepackage{mathtools}
\usepackage{cleveref}
\usepackage[ruled,vlined]{algorithm2e}
\SetKwInput{KwInput}{Input}
\SetKwInput{KwOutput}{Output}
\SetKwInOut{Parameter}{Parameter}

\usepackage{cases}
\usepackage{lineno}
\usepackage{setspace}
\DeclareMathAlphabet\mathbfcal{OMS}{cmsy}{b}{n}

\newcommand{\floor}[1]{\lfloor #1 \rfloor}

\usepackage[T1]{fontenc}
\usepackage[utf8]{inputenc}
\usepackage{charter}
\usepackage{environ}
\usepackage{tikz}
\usetikzlibrary{calc,matrix,babel}

\makeatletter
\let\matamp=&
\catcode`\&=13
\makeatletter
\def&{\iftikz@is@matrix
  \pgfmatrixnextcell
  \else
  \matamp
  \fi}
\makeatother

\newcounter{lines}

\newcounter{vtml}
\newif\ifvtimelinetitle
\newif\ifvtimebottomline
\tikzset{description/.style={
  column 2/.append style={#1}
 },
 timeline color/.store in=\vtmlcolor,
 timeline color=red!80!black,
 timeline color st/.style={fill=\vtmlcolor,draw=\vtmlcolor},
 use timeline header/.is if=vtimelinetitle,
 use timeline header=false,
 add bottom line/.is if=vtimebottomline,
 add bottom line=false,
 timeline title/.store in=\vtimelinetitle,
 timeline title={},
 line offset/.store in=\lineoffset,
 line offset=4pt,
}

\NewEnviron{vtimeline}[1][]{%
\setcounter{lines}{1}%
\stepcounter{vtml}%
\begin{tikzpicture}[column 1/.style={anchor=east},
 column 2/.style={anchor=west},
 text depth=0pt,text height=1ex,
 row sep=1ex,
 column sep=1em,
 #1
]
\matrix(vtimeline\thevtml)[matrix of nodes]{\BODY};
\pgfmathtruncatemacro\endmtx{\thelines-1}
\path[timeline color st] 
($(vtimeline\thevtml-1-1.north east)!0.5!(vtimeline\thevtml-1-2.north west)$)--
($(vtimeline\thevtml-\endmtx-1.south east)!0.5!(vtimeline\thevtml-\endmtx-2.south west)$);
\foreach \x in {1,...,\endmtx}{
 \node[circle,timeline color st, inner sep=0.15pt, draw=white, thick] 
 (vtimeline\thevtml-c-\x) at 
 ($(vtimeline\thevtml-\x-1.east)!0.5!(vtimeline\thevtml-\x-2.west)$){};
 \draw[timeline color st](vtimeline\thevtml-c-\x.west)--++(-3pt,0);
 }
 \ifvtimelinetitle%
  \draw[timeline color st]([yshift=\lineoffset]vtimeline\thevtml.north west)--
  ([yshift=\lineoffset]vtimeline\thevtml.north east);
  \node[anchor=west,yshift=16pt,font=\large]
   at (vtimeline\thevtml-1-1.north west) 
   {\textsc{Timeline \thevtml}: \textit{\vtimelinetitle}};
 \else%
  \relax%
 \fi%
 \ifvtimebottomline%
   \draw[timeline color st]([yshift=-\lineoffset]vtimeline\thevtml.south west)--
  ([yshift=-\lineoffset]vtimeline\thevtml.south east);
 \else%
   \relax%
 \fi%
\end{tikzpicture}
}

\title{A Mean Field Game Analysis of Consensus Protocol Design } 

\author{
 Lucy Klinger 
\\ \small {Beijing International Center for Mathematical Research, Peking University, Beijing, 100871, China}
\\ {\small \href{mailto:lucyk@bicmr.pku.edu.cn}{lucyk@bicmr.pku.edu.cn} }
\and \hspace{-0.8cm} Lei Zhang 
\\ \footnotesize{ Beijing International Center for Mathematical Research, Center for Quantitative Biology, Peking University, Beijing 100871, China }
\\ {\small \href{mailto:zhangl@math.pku.edu.cn}{zhangl@math.pku.edu.cn}}
\and Zhennan Zhou 
\\ \small { Beijing International Center for Mathematical Research, Peking University, Beijing, 100871, China}
 \\ {\small \href{mailto:zhennan@bicmr.pku.edu.cn}{zhennan@bicmr.pku.edu.cn} }
 }

\date{}

\providecommand{\keywords}[1]{\textbf{\textit{Keywords:}} #1}

\makeindex

\begin{document}
\addtocontents{toc}{\protect\setcounter{tocdepth}{2}}
\maketitle

\n \keywords{
mean field game,
Hamilton-Jacobi-Bellman equation,
Fokker-Planck equation,
decentralized blockchain, 
mechanism design,
consensus protocol,
Proof-of-Work
}

\begingroup
\let\clearpage\relax

\abstract 

\n A decentralized blockchain is a distributed ledger that is often used as a platform for exchanging goods and services. This ledger is maintained by a network of nodes that obeys a set of rules, called a consensus protocol, which helps to resolve inconsistencies among local copies of a blockchain. In this paper, we build a mathematical framework for the consensus protocol designer, specifying (a) the measurement of a resource which nodes strategically invest in and compete for to win the right to build new blocks in the blockchain; and (b) a payoff function for such efforts. Thus, the equilibrium of an associated stochastic differential game can be implemented by selecting nodes in proportion to this specified resource and penalizing dishonest nodes by its loss. This associated, induced game can be further analyzed using mean field games. The problem can be broken down into two coupled PDEs, where an individual node’s optimal control path is solved using a Hamilton-Jacobi-Bellman equation, and where the evolution of states distribution is characterized by a Fokker-Planck equation. We develop numerical methods to compute the mean field equilibrium for both steady states at the infinite time horizon and evolutionary dynamics. As an example, we show how the mean field equilibrium can be applied to the Bitcoin blockchain mechanism design. We demonstrate that a blockchain can be viewed as a mechanism that operates in a decentralized setup and propagates properties of the mean field equilibrium over time, such as the underlying security of the blockchain.



\newpage
\section{Introduction}\label{section_1}

\n A {\it blockchain} is a decentralized ledger that \hypertarget{fig_identities}{
records sequences of real-time transactions, denoted by $\{ T_t \}_{t \in [0, \infty)}$}, representing assets ownership
at time $t$ \citep{nakamoto2008bitcoin, atzei2018formal}. 
Blockchains are most often used as platforms for exchanging goods and services \citep{kim2020advanced}.
With no reliance on a trusted central authority, 
a blockchain is maintained by \hypertarget{fig_consensus_nodes}{a set of nodes, indexed by $i \in \{1,2,\cdots,{M}\}$, in a {\it decentralized network}, as shown in the bottom right corner of \autoref{Blockchain_structure}. 
This decentralized network consists of vertices and edges, 
where vertices represent nodes and edges represent communication between nodes.}
 \begin{figure}[h]
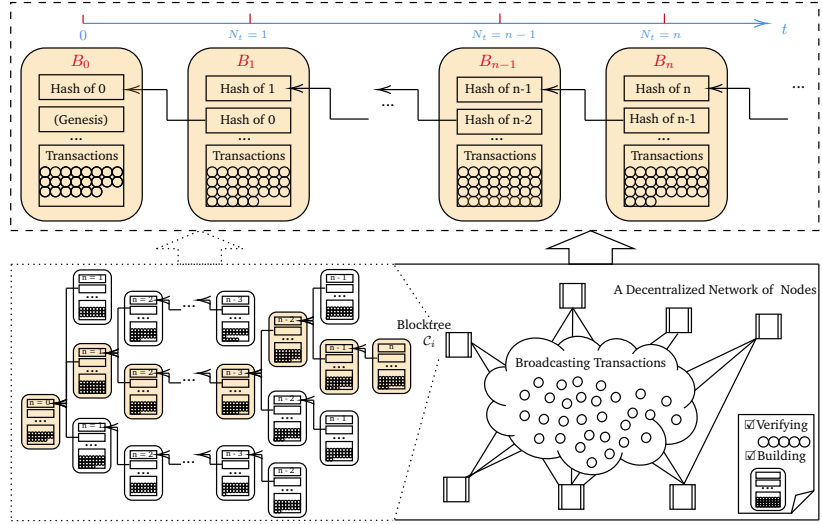

 \centering
\resizebox{.68\textwidth}{!}{

\tikzset{every picture/.style={line width=0.75pt}} 


}
 \caption{An Overview of Blockchain Mechanism}
 \label{Blockchain_structure}
 \end{figure}
Those nodes record transactions by building a sequence of valid discretized blocks denoted by $\mathcal{B}$.
A {\it block} $B(n) \in \mathcal{B}$ with a block height $n$ consists of a block header which includes metadata, 
together with a sequences of timed transactions not found in previous block.
\begin{equation}\label{block} 
B(n) := \big\langle ~ \underbrace{\text{\scriptsize version}  ,~ \text{\scriptsize difficulty} ,~ \text{\scriptsize nonce} ,~  \text{\scriptsize previous block's hash}  ,~ \text{\scriptsize timestamp} ,~ \text{\scriptsize Merkle root}} _{\text{Block Header 
}},~  \underbrace{\{ T_t ~|~ T_t \not\in B(n'), \forall n'<n \}}_{\text{Block Body}} ~~ \big\rangle
\end{equation}
Here, {\it version} is what is followed by the consensus protocol node, while {\it  difficulty} and {\it nonce} refer to their respective use in the protocol design. {\it Timestamp} is when a block was built, and {\it Merkle root} is used for quick validation to retrieve complete transaction records
$\{T_t \in B(n') , \forall n'<n \}$.
Each block $B(n)$ is linked to the block $B(n-1)$ with a hash pointer $\Lsh$, 
as denoted by $B(n-1) \Lsh B(n)$. 
This notion represents the previous block’s hash information when it is stored in the current block’s header, as shown at the top of \autoref{Blockchain_structure}.
Such block arrival process can be modeled as a non-homogeneous Poisson process $\{N_t, t \geq 0\}$ with intensity $\int_0^{t} {{\lambda}} (s) \diff s $ \citep{fralix2020classes}
and with a probability mass function that is is given by 
\begin{equation}\label{def_poisson_intensity}
\Prob \{N_t =n\}={\frac {\big(\int_0^{t} {{\lambda}} (s) \diff s \big)^{n} }{n!}}e^{-\int_0^{t} {{\lambda}} (s) \diff s }.
\end{equation}
In a decentralized structure,
\hypertarget{fig_replica}{each node has its own local replica of the {\it blocktree} $ \mathcal{C}_i (n)$, as shown in the bottom left corner of \autoref{Blockchain_structure}.
A blocktree $ \mathcal{C}_i (n) \subset (\mathcal{B},\{\Lsh\})$ is a directed graph
that is made up of a subset of all possible valid block paths from a genesis block $B(0)$, such that for each valid block $B(n) \in \mathcal{B}$, there is exactly one path from $B(0)$ to $B(n)$ in a blocktree ${ \mathcal{C}_i} (n)$ of up to block height $n$}.
For a given block tree $\mathcal{C}_i (n)$ with a maximum height $n$,
our ultimate goal is to create a {\it blockchain} that has a fully validated ledger of transactions, 
which is formally defined as the path 
\begin{equation}\label{blockchain}
 \Big\{ B(0)  \Lsh B(1)  \Lsh \dots  \Lsh B(n) ~ \Big| ~ B(n) \in \mathcal{B} \Big\}  \subseteq\mathcal{C} (n) \,,
\end{equation}
where, the genesis block $B(0)$ to the current block 
$\{ B(n) |  B(n) \in \mathcal{B} \}$, a correct node is proposed.

Often, a variety of orphan blocks represents a latency in the network, 
which occurs when nodes broadcast blocks onto the network simultaneously. 
Due to this broadcasting delay, such nodes may observe different versions of the blocktree, 
as denoted by the collection of blocktrees
\begin{equation}\label{blocktree_set}
\mathbfcal{C}  = {\bigcup\limits_{i=1}^{{M}}  \mathcal{C}_{i} (n)} \,.
\end{equation}
Theoretically, the concept of a blocktree's appearance is based on a subjective consensus \citep{narayanan2016bitcoin}.
Therefore, to resolve inconsistencies among individual nodes, 
we need a set of rules  
\begin{equation}\label{consensus_rule}
{g}:  
\mathbfcal{C}  \rightarrow \mathcal{C} (n) \,,
\end{equation}
called the {\it consensus protocol},
so that every node can update its local replica of the blocktree to match the one proposed by a correct node \citep{zhang2020modeling}.

It has been observed that 
when a transaction receives more confirmations, 
a block containing that transaction will have a greater number of other blocks pointing to it. As a result, this transaction has an increased likelihood of being included in the blocktree proposed by a correct node. Using the equilibrium of a noncooperative game can help the consensus protocol to achieve this outcome. Instead of picking a random node, we can build a game environment that allows a large number of risk neutral nodes to compete for the right to build block $B(n)$, so that new transactions are verified by numerous confirmations derived from competition between nodes (a process normally referred to as   \emph{mining a block}) \citep{liu2019survey}. 
Apart from the consensus protocol, nodes do not face any other strategic restrictions. Rational nodes act in order to maximize their own utility. Malicious nodes, on the other hand, may launch attacks that damage blockchain networks. A game environment will allow the mechanism to approximately select correct nodes. To accomplish this, a stochastic differential game model can be built to operate within a decentralized setup
\citep{dhamal2019stochastic}.

\begin{game}[Stochastic Differential Game]\label{dynamic_game_model}
A stochastic differential game $\gamma$ for ${M}$ nodes, played in real time,
is a 4-tuple 
\[
\gamma = \Big\langle  \{1,2,\cdots,{M}\}, {\Omega}, \vO, \vpi \Big\rangle \,,
\]
which
consists of the following components:
\begin{enumerate}[label=(\Roman*), font=\upshape] 
\item a finite set of rational and intelligent non-cooperative nodes indexed by $i \in \{1,2,\cdots,{M}\}$
\item 
a filtered probability space $\big(\Omega, \mathcal{F}, \Prob \big)$ 
where the filtration is denoted by $\left(\mathcal{F}_{t }\right)_{0 \leq t \leq \infty}$ of $\upsigma$-fields on $(\Omega, \mathcal{F})$ such that $\mathcal{F}_{t'} \subset \mathcal{F}_{t} \subset \mathcal{F}, \forall t' \leq t $, 
which represents the information available \footnote{Here we borrow the notation from \citep{marek2017}} to all nodes at time $t$.
Let $\big(x_{i}(0) , b_{j}(0) \big)$ be a random variable known up to a $ \Prob$-null set and let $W$ and $ \mathcal{P}$ be a
Brownian motion and a Poisson random measure with intensity $\lambda$, respectively. 
Assume that $W$ and $ \mathcal{P}$ are independent of $\mathcal{F}_0$. 
Each node has a $2$-dimensional stochastic random state process 
\[
\Big(x_{i}(\cdot) , b_{j}(\cdot) \Big) \in X_{i} \times Y_{j}  \subset \R^{2} \,,
\] 
adapted to filtration $\mathbb{F} :=\left(\mathcal{F}_{t}\right)_{0 \leq t \leq \infty}$, 
where $x_{i}(\cdot) $ represents certain state dynamics (e.g. a wealth process) and $b_{j}(\cdot) $ represents the token price process.

\item  a set of all possible outcomes $\vO = \displaystyle {\prod_{i=1}^{{M}} O_{i}}$, where $O_i$ is the set of admissible controls for node $i$
\begin{equation}
\displaystyle O_i := \Big\{ \alpha_i (\cdot) : {\Omega} \to \R \mbox{ measurable }~ \forall t \Big\} \, 
\end{equation}
\item \label{optimal_control_pro}
 a set of all possible payoffs $\vpi = (\pi_{1},\dots,\pi_{{M}})$, where node $i$ chooses the optimal control $\alpha_i$ to maximize its running payoff functional over time $t \in [0,\infty)$
\begin{equation}\label{optimal_control_2}
 \pi_{i} \Big(\alpha_i, [0,\infty) \Big) := 
\E \left[ \int_{0}^{\infty} u_{i} \Big(  x_{i}(t), b_{j}(t), {\alpha_{i}(t)}, \vx_{-i}(t), \vb_{-j}(t), {\valpha_{-i}(t)} \Big)    e^{-r t} \, \diff t \right]
\end{equation}
over the action taken by all nodes.
We use ${{-i}}$ to denote all nodes except node $i$, while $r > 0$ represents a fixed discount factor.
Here, the marginal payoff function 
is associated with some controlled evolution, 
which is driven by the Markovian feedback control
 jump-diffusion process
\begin{equation}
 \diff x_i (t) = \Big( r x_t - c({\alpha_i({t})} ) \Big)  \diff t\ + b_{j}(t) k_t \diff N_t  ,,
\end{equation}
where $c$ is a cost function; 
the {reward} that arrives at system $\mathcal{P} ( \vO \times \vX )$ 
is calculated according to a non-homogeneous Poisson process, 
as described in \autoref{def_poisson_intensity}.
\end{enumerate}
\end{game}

Nodes in the network are coupled through payoffs and state dynamics. 
Their strategic interactions, as well as the evolution of the above game, can be analyzed using concepts from game theory.
Within the context of a stochastic differential game $\gamma$,  
nodes can learn from and predict each other's behaviors, 
and then develop optimal reaction strategies based on equilibrium analysis. 
The set of optimal strategies that achieve a Nash equilibrium in every state of the stochastic differential game is named a Markov perfect equilibrium
\citep{shapley1953stochastic}.
\begin{definition}[Markov Perfect Equilibrium (MPE)] \label{Markov_Perfect_Equilibrium}
An ${M}$-tuple of feedback control functions
\[
\hat \valpha  = \Big( \hat \alpha_{1}, \hat \alpha_{2}, \ldots , \hat \alpha_{{M}} \Big)
\]
constitutes a Markov Perfect Equilibrium (MPE) for Stochastic Differential Game \ref{dynamic_game_model} 
if the feedback control ${\hat \alpha}_{i}(\cdot)$
provides a solution to the Optimal Control Problem \ref{optimal_control_pro}
for the $i$-th node for $\forall i \in \{1,2,\cdots,{M}\}$
subject to the controlled dynamics of the system for every initial conditions $x_i(0)$.

\end{definition}

As time evolves in a game, it is assumed that nodes will typically place somewhere near Equilibrium \ref{Markov_Perfect_Equilibrium}. 
We aim to design a game that selects nodes in proportion to the ownership of a substantial resource which cannot be monopolized and also functions as a tamper-proof summary of large data files. To achieve this, we need to impose a surjective mapping
\begin{equation}\label{measurement}
\Psi : \mathbfcal{C} \to \vO
\end{equation}
between blocktrees and the outcome of Game \ref{dynamic_game_model},
which guides 
nodes to invest strategically in a substantial resource 
so they can compete to build new blocks (\autoref{block}) in the blockchain  (\autoref{blockchain}).
Such a design would penalize dishonest nodes through loss of resources, 
thereby discouraging nodes from misbehaving or launching attacks. 
For example, this substantial resource can be, among others, a measurement of work done 
 (used in a Proof-of-Work protocol) or a measurement of ownership (used in a Proof-of-Stake protocol)
\citep{wang2019survey}.
\begin{definition}[Consensus Protocol]\label{Consensus_Protocol_Mechanism}
A consensus protocol in blockchain mechanism for $\gamma$ (described in Game \ref{dynamic_game_model}) is a 4-tuple 
\[
\mathfrak{C} = \Big\langle \mathbfcal{C}, \Psi, \vpi, g \Big\rangle, 
\]
which
consists of the following components:
\begin{enumerate}[label=(\Roman*), font=\upshape] 
\item a set of different local replicas of blocktrees $\mathbfcal{C}  = {\bigcup\limits_{i=1}^{{M}}  \mathcal{C}_{i}}$  , as described in \autoref{blocktree_set}
\item a surjective mapping $\Psi: \mathbfcal{C} \to \vO$, as described in \autoref{measurement}, which is designed to allow nodes to invest strategically in a substantial resource, so they can compete to create new blocks $B(n)$ with the increase of height $n$ in the blockchain, as described in \autoref{blockchain}

\item nodes that are rewarded for their efforts, which is reflected in payoff function $\vpi$, as described in \autoref{optimal_control_2}

\item 
a set of rules that becomes
${g} \big( \Psi^{-1}(\vO) \big)$, as described in \autoref{consensus_rule}, 
so that every node can update its local replica of blocktree to the one chosen by the outcome of  game $\gamma$,
resulting in a consensus $\mathcal{C} (n)$ among all nodes.

\end{enumerate}
\end{definition}
Based on implementation theory \citep{jackson2001crash}, the equilibrium $\hat \valpha \in \vO$ can then be utilized to develop an incentive scheme within a consensus protocol, which nodes participate in and contribute resources to for the purpose of maintaining the blockchain \citep{swan2015blockchain}. 
This equilibrium is what governs the underlying security of the blockchain.
\begin{problem}\label{design_problem}
Given a game environment $\gamma = \Big\langle  \{1,2,\cdots,{M}\}, {\Omega}, \vO, \vpi \Big\rangle$
as described in Definition \ref{dynamic_game_model},
the designer needs to specify two mappings $\Psi$ and $\vpi$ in the consensus protocol $\mathfrak{C} = \Big\langle \mathbfcal{C}, \Psi, \vpi, g \Big\rangle
$ as described in Definition \ref{Consensus_Protocol_Mechanism},
so that 
an equilibrium 
$
\hat \valpha = \Psi(\mathbfcal{C} )
$
of the game of incomplete information 
\[
\Gamma = \Big\langle  \{1,2,\cdots,{M}\},{\Omega}, \Psi(\mathbfcal{C} ), \vpi \Big\rangle
\]
can be implemented for all states $\vx(\cdot) \in {\Omega} $ over ${M}$ and $\vO$.
\end{problem}

Within this framework, we aim to establish the existence of the Nash equilibrium described above. 
Without loss of generality, we use the Proof-of-Work (PoW)
Protocol in the Bitcoin blockchain as an example to demonstrate the existence of such an equilibrium. 
Note that more complex schemes, such as the Proof-of-Stake (PoS) mechanism \citep{bentov2014proof}, 
can be substituted for the PoW analysis used in our work.

\subsection{
Proof-of-Work 
Protocol}\label{example_POW}

As discussed in Problem \autoref{design_problem}, 
we aim to use two mappings, $\Psi$ and $\vpi$ (as described in Definition \ref{Consensus_Protocol_Mechanism}),
to demonstrate the existence 
of an equilibrium in the Bitcoin mining game.
In this subsection, 
we first aim to describe the two mappings, $\Psi$ and $\vpi$, as specified in the PoW protocol  of the Bitcoin
blockchain.

In the PoW protocol, the game selects nodes in proportion to the amount of work done by the computing power that is allocated to building blocks (see \Cref{Appendix_POW} for more details).
The surjective mapping $\Psi$, as described in \autoref{measurement},  is now replaced by a measurement of work done, 
i.e., the hashrate function $\alpha$, which represents the expected computing power needed for nodes to compete for the right to create new blocks $B(n)$ with an increase of height $n$ in the blockchain, as described in \autoref{blockchain}. 
Consequently, every node adheres to the longest chain \citep{nakamoto2008bitcoin} and updates its local blocktree replica to match the blocktree that maps onto the game’s outcome, representing the maximum computing power used:
\begin{equation}\label{choose_maximum_resources}
{g} \big( \Psi^{-1}(\vO) \big) := \Big\{ C_j ~\Big| ~ j = \hspace{-0.1cm} \argmax_{i \in  \{1,2,\cdots,{M}\}} \valpha \Big\} \,. 
\end{equation} 
Such longest chain rule design can probabilistically prohibit faulty node behaviors, including malicious attacks, node mistakes, and connection errors \citep{wang2019survey}.

The payoff function $\vpi$ is designed to offer rewards within a blockchain so that nodes participate in the game as well as control reward amounts that have an exponential decay design (as demonstrated in \autoref{fig_Ratio_of_k_K}):
\begin{equation}\label{inflation_rate}
\text{the inflation rate} = \frac{\text{tokens per block} \times \text{block per second} }{\text{total circulation}} \,.
\end{equation}
This inflation rate design can lead to a fixed finite token supply, which is a necessary condition for achieving a meaningful just price of money \citep{roets1991bernard}. Such an aim can be achieved by choosing a counting process for deciding the number of token rewards \citep{bowden2018block},
as demonstrated in \autoref{fig_reward_number}.
This reward process is summarized as follows:
\begin{enumerate}[label=(\Roman*), font=\upshape] 

\item \label{reward_mechanism_1} the intensity of the block arrival process is set to be (see \autoref{Appendix_POW} for detailed derivations)
\begin{equation}\label{expected_block_arrival_rate}
\left({{{\lambda}}}_t \right)_{t \in (t_{2016(s-1)}, t_{2016s}]} = \frac{2016}{t_{2016s}-t_{2016(s-1)}}
=  \frac{1}{600} \left( \frac{{  1- \sqrt[M]{1- \frac{2016}{{H_{s-1}}} } }}{1- \sqrt[M]{1- \frac{2016}{H_s}} }\right) \,,
\end{equation}
where the total hashes over time segment $(t_{2016(s-1)}, t_{2016s}]$ is defined as
\begin{equation}\label{initial_hash}
 H_s := {\int_{t_{2016(s-1)}}^{t_{2016s}} h(t) \diff t} 
\qquad \qquad
\mbox{with} 
\qquad \qquad
 H_0 := \frac{2016}{1-\left(1-\frac{1}{2^{32}} \right)^{M}} \, . \qquad
\end{equation}
and the total hashrate $h_t$ at time $t$ is the stochastic variable
\[
h(t) := \sum_{i \in  \{1,2,\cdots,{M}\}} \alpha_i (t)  \,,
\]
which represents the total computational power nodes invest in, as well as the overall confidence or sentiment regarding the blockchain system
\item \label{reward_mechanism_2} 
the number of token in each block at time $t$ is
\begin{eqnarray}\label{number_of_block_reward}
k_t 
= 50 \left(\frac{1}{2}\right)^{l} + \bar k \,,
\end{eqnarray}
where each time interval $l \in \{1, . . . , 32 \}$ includes 210000 blocks:
\[
l = \floor{\frac{2016 N_t}{210000}},
\quad
\mbox{omitting the small transaction fees }
\bar k \to 0  \, .
\]
Hence the number of cumulated tokens in circulation is 
\begin{equation}\label{number_of_cumulated_tokens}
{K_t} = \left(210000 \times 50 \right) \left({\frac {1-(\frac{1}{2})^{l}}{1-\frac{1}{2}}}\right)+ \left(2016-\hspace{-0.3cm}\mod{(210000 l,2016)}\right)\left(50\left(\frac{1}{2}\right)^{l}\right) \,.
\end{equation}
Note that ${k_t} $ and ${K_t}$ are both measured with block arrival process $N_t$ (as described in \autoref{def_poisson_intensity}).
\end{enumerate}

\begin{figure}[H]
\centering
{\includegraphics[width=0.62\linewidth]{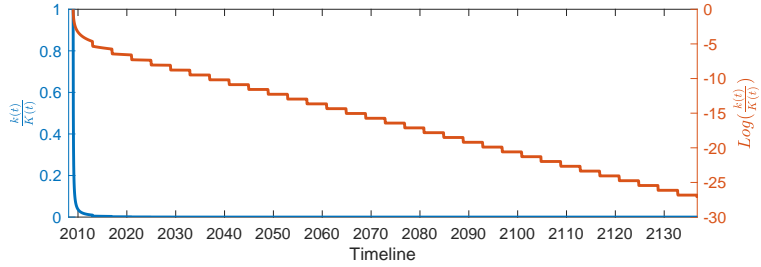}}
\caption{
The token’s inflation rate in the Bitcoin blockchain is plotted over the timeline.
As described in \autoref{inflation_rate}, the ratio of the number of tokens created over time to the total number of tokens in circulation,
  which starts at 1 and decays exponentially until it reaches 0.
This curve is controlled by the design of Mechanisms \ref{reward_mechanism_1} --\ref{reward_mechanism_2}. 
}
\label{fig_Ratio_of_k_K}
\end{figure}

\vspace{-0.3cm}
\begin{figure}[H]
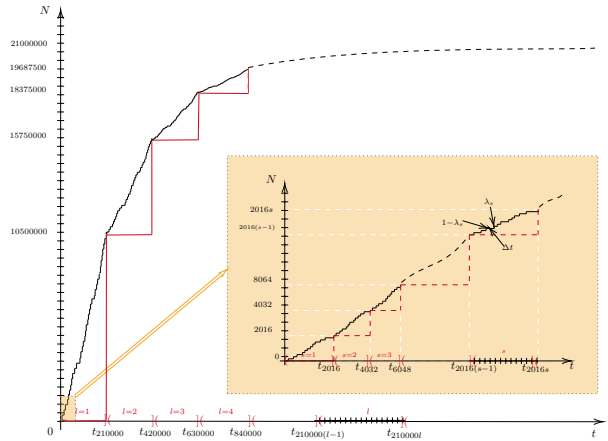

  \centering
\resizebox{0.5\textwidth}{!}{

\tikzset{every picture/.style={line width=0.75pt}} 


}
 \caption{
 The block arrival process modeled as a counting process. The mining difficulty adjusts every 2016 blocks, while each block reward starts at 50 tokens per block and halves every 210,000 blocks.
}
  \label{fig_reward_number}
\end{figure}

Modern mathematical financial theory postulates that perfect markets should have no arbitrage opportunities available.  Assuming an arbitrage-free token pricing model, the formation of total token value should be converted from the total cost of resources used in the production through the mining process \citep{hayes2019bitcoin}.
\begin{assumption}[Token Price]\label{assumption_token_price}
We assume that the market value of a token is an Ornstein-Uhlenbeck process (OU-process)
$\left\{b_{t}\right\}_{t \in [0,\infty)}$, which is driven by 
an equilibrium level of multiple of its cost of production
\begin{equation}
\diff b_t =  (\hat b_t - b_t) \diff t +\sigma \diff W_t \,,
\end{equation}
where $W_{t}$ is the standard Wiener process, and a multiple of the cost of a token production,
\begin{equation}\label{token_price}
\hat b_{t}= \beta \frac{c(h_t)}{K_t} \,,
\end{equation}
is the ratio of the blockchain production cost to the number of cumulated tokens in circulation $K_t$  (as described in \autoref{number_of_cumulated_tokens}),
which acts as an equilibrium level for the token price process: 
if the current market value of the  token process is less than the production price $\hat b_{t}$, the drift will be positive; 
if the current market value of the  token process is greater than production price $\hat b_{t}$, the drift will be negative. 

The parameter $\beta$ can be estimated from the curve fitting tool, as demonstrated in \autoref{token_price_fig}.
We also assume that the Poisson process is independent of the Wiener process.
\end{assumption}

\begin{figure}[H]
  \centering
  \includegraphics[width=0.5\linewidth]{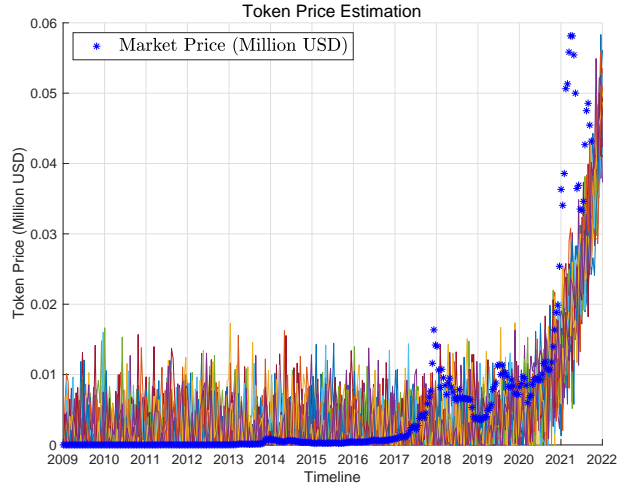}
\vspace{-0.2cm}
 \caption{
 The value of each reward is the token market price, labeled in blue stars, which is driven by supply and demand in the marketplace. Fluctuations in supply and demand can exhibit large spikes or drops, which are influenced by factors, including market information, regulations governing its sale, speculation, consumer sentiment, and the rumor mill \citep{chu2015statistical}. Here, we assume the market price of the token has a mean reversion OU-process, which can be shown as a multiple of its cost of production (labeled in red). This parameter $\beta = 2 \times 10^4$ is estimated by using the curve fitting tool, while simulating 18 sample paths and the volatility is set to be $\sigma = 0.005$.}
  \label{token_price_fig}
\end{figure}

Overall, Problem \ref{design_problem} boils down to an economics question: whether it is profitable for nodes to mine under the conditions provided by the PoW protocol \ref{example_POW} and the token price Assumption \ref{assumption_token_price}. 
We can explore the answer by analyzing the equilibrium $\hat \valpha$ of the game of incomplete information $\Gamma$ (in Problem \ref{design_problem}) and substituting the function $\pi$ specified in Mechanisms \ref{reward_mechanism_1} -- \ref{reward_mechanism_2}.
%

\subsection{Related Works and Main Results} 

Naturally, game theory approaches have been applied to problems of incentive mechanism design in blockchain consensus protocols. 
There are a number of published surveys on game theoretical approaches to blockchain-related issues. 
\cite{liu2019survey} reviews game models proposing to address common issues in blockchains, 
including security, mining management, and economic and energy trading. 
\cite{wang2019survey} provides a comprehensive survey on consensus mechanisms in blockchain backbone networks and strategies for self-organization by individual nodes from a game-theoretic point of view. 
\cite{sadek2020blockchain} presents a wide range of consensus algorithms and uses a comprehensive taxonomy to examine their properties and implications in tabular formats and other illustrations. 
The mechanism design of consensus protocol algorithms has also been discussed in \citep{weeks2018evolution} and other forums.

Finding an MPE in Definition  \ref{Markov_Perfect_Equilibrium} is complex, 
but mean field games are known to approximate induced Stochastic Differential Games \ref{dynamic_game_model} well with a large number of nodes. 
Mean field game theory was first created to study the concept of the MPE in stochastic differential games where the number of players tends to infinity. Introduced in \citep{lasry2006jeux, lasry2007mean}, a similar approach was formulated independently in \citep{huang2006nash}.
Successive online courses at the Coll{\'e}ge de France \citep{lions2007theorie} and lecture notes \citep{gueant2011mean} presented a number of developments, covering topics such as the structures, concepts, and existence of equilibrium; links with Nash equilibria and dynamics in n-player game theory (when n tends to infinity); variational principles for decentralization; and notions of stability in solutions. 
Since these concepts have been extensively developed, other authors have made further contributions in various applications, notably in economics \citep{achdou2014partial, achdou2017income}. 
If each node is considered to be small, thus having a negligible impact on the mean hashrate, there is a convergence of the MPE to a weak MFE in open-loop cases. In some partial results and examples, weak MFG equilibria can arise as the limit of the MPE \citep{lacker2018convergence}. From a numerical perspective, a finite difference method has been applied to approximate the mean field equilibrium in several studies \citep{achdou2010mean, achdou2013mean, achdou2016convergence, achdou2020mean}.

In our paper, we aim to test the modified mean field game design of Stochastic Differential Games \ref{dynamic_game_model} with multi-dimensional dynamic formulations of some state dynamics (e.g., a wealth process and a token price process). 
A few papers study blockchain technology adoption through the lens of mean field game theory.
\cite{li2019mean} studies the behavior of the mean field game equilibrium in mining games, in which nodes compete against each other for block rewards through the investment of computational power (hashrate). Under this model, nodes are characterized by a one-dimensional wealth state, and they may choose to invest in hashrate to maximize their expected utility within a fixed time horizon. We extend this framework by introducing an OU-process token price model that is driven by an aggregate quantity (total hashrate). We also relax the time horizon, making it infinite, by allowing the blockchain to adjust the jump intensity of the block arrival process by the total hashrate in real time. As a result, the wealth process in our mean field game is a non-homogenous compound poisson process. Thus, we need to extend the dimension of the state space, since the jump size depends on the token price model we proposed.
In contrast to \cite{li2019mean}, we choose a marginal utility that depends on the control variable, which is also investigated in \cite{bertucci2020mean}. \citeauthor{bertucci2020mean} uses the solution of a master equation to characterize an equilibrium of the mining game for one individual unit of computing power devoted to mining for the blockchain. This master equation approach is robust when small changes are made in modeling assumptions: it preserves the essential structure of the game across all enrichment of various models. By contrast, we take a different approach by focusing on the Hamilton-Jacobi-Bellman equation, which is integral to the solution of the master equation with respect to the discounted total hashrate \citep{bertucci2020mean}. Our focus allows us to introduce aggregate quantities (the total hashrate) from microscopic decisions (investing in computational powers). Our approach also has the advantage of using the forward-backward system with both the Hamilton-Jacobi-Bellman and the Fokker-Planck equation, which is already well developed in the mathematical literature. Further, the value function for the optimal control problem and the evolution of density function has a nice real-world interpretation.

Our work contributes to the consensus protocol design of the core blockchain mechanism. 
First, we build a more realistic framework, using Proof-of-Work (PoW) in a Bitcoin blockchain as an illustrative example. 
Then, we discuss the security of the blockchain design that employs a mining incentive at the equilibrium. 
We approach the intrinsic mechanism design problem by implementing a mean field game equilibrium. 
From the resulting numerical solution, we can predict long-run steady states at an infinite time horizon as well as short-run evolutionary dynamics. 
We show that the game arising from the PoW mechanism leads to an equilibrium and the consensus protocol is well-posed. 
Consequently, we use this idea to analyze and update the rules for the existing consensus protocol scheme in real-world situations, as well as design a new consensus protocol scheme.

The remainder of this paper is organized as follows.
In \Cref{section_1}, we introduce a basic model formulation of the blockchain mechanism and raise a consensus protocol design problem. 
\Cref{section_2} uses a mean field approximation to analyze the behavior of the equilibrium and develop an incentive scheme for the consensus protocol design problem introduced in \Cref{section_1}. 
\Cref{section_4} presents the numerical results and discussion.
The last section concludes with some remarks on possible future directions.

\section{Mean Field Game Analysis}\label{section_2}

In a mean field game, one can describe population behavior as a distribution over the state space rather than in terms of specific node's identification. 
Each node optimizes its payoff by solving a dynamic programming problem, 
assuming the distribution of nodes is specified; 
on the other hand, if every node uses its optimal control, we can better 
infer the distribution of nodes. 
Overall, if we start with a distribution and compute the best response strategy for every node, we will obtain a new stationary distribution in population states. 
The fixed point of this mapping is called a Mean Field Equilibrium (MFE). 
In this section, we give a formal definition for an MFE in the PoW protocol, 
and then 
demonstrate the existence of the MFE in next section.

First,
consider a histogram that counts the number of nodes which fall under each of the disjoint initial wealth-token price categories. 
Let $m (t, x, b)$ be such a density function, denoting the fraction of nodes whose state is $(x, b) \in\R^2$ at time $t$, i.e.,
\begin{equation}\label{densitb_def} 
m (t, x, b) \approx  \frac{1}{M(t)} \sum _{i }\1_{\big\{x_{i }(t) =x, ~b_{j}(t) =b \big\}} \,.
\end{equation}
\autoref{densitb_def} can then be interpolated by a continuous curve of smooth frequencies over the bins of the histogram, normalizing it to $1$. 
The notion of a continuum of nodes captures scenarios in which the interaction between nodes is a probability distribution of random variables that fall within a particular range of values in their state space, so that each node is labeled by its two-dimensional wealth-token price state, which is independent of the node's identity. 
 Such mathematical treatment is widely used for exploring the collective behavior of interacting agents in physics and biology \citep{bolleystochastic, liu2021investigating, liu2020rigorous}.
It also has a wide range of application in mathematical finance
\citep{casgrain2020mean, mkiramweni2019mean, nie2018fair}.
 For ease of notation, we omit dependency on subscripts $i$ and $j$ for the remainder of the paper.

\begin{game}[Mean Field Game]\label{general_structure}
The distribution of nodes is now defined as a continuously differentiable density function, $m(t, x, b)$, which satisfies
\begin{equation}\label{eqn:initial_density}
\int_{X} \int_{B} m(t, x, b) \diff x \diff b= 1 \quad \forall t \in [0,\infty)\,,
\end{equation}
subject to an initial condition $m(0, x, b)$.
Each node is labeled by the two-dimensional wealth state $x \in X$ and the token price state $b \in Y$, 
with boundary conditions at $\min(X)=0$ and $\min(Y)=0$ for nonnegative wealth-token price state space constraints.

\begin{enumerate}[label=(\Roman*), font=\upshape] 

\item \label{discounted_hashrate}
Each node invests in the hashrate to participate in the mining game. 
Because it is cheaper to buy the same computational power over time, a node can acquire more hashes per second for the same cost. 
Here we focus on the feedback control process $\alpha_{t} (x, b)$,
the hashrate discounted by the rate of the technological progress at time $t$ (i.e., the notion of real hashrate as found in \cite{bertucci2020mean}). 
Let the set of the admissible Markovian feedback control process be denoted by
 \begin{equation}\label{control}
\left\{\alpha_{t} (x, b): [0,\infty) \times X \times Y \rightarrow \mathbb{R}  \mid \alpha_{t} \text { is adapted to the filtration }\left\{\mathcal{F}_{t}\right\}_{t \geq t'}\right\}
\end{equation}

\item 
Each node aims to maximize its expected total utility over the infinite-time horizon $[0,\infty)$ 
\begin{equation}
\label{expected_total_utility}
\max_{\alpha_{t}}\E\left[\int_0^{\infty} e^{-{{r}} t} u(t, x_{t}, b_t, \alpha_{t}) \diff t \, \Big| \, x(0) = x_{0} , b(0)=b_0\right] \,,
\end{equation}
where
$r > 0$ is an instantaneous interest rate,
and the marginal utility function is assumed to be strictly concave with the control variable
\begin{equation}
\label{marginal_utility}
{u}: [0,\infty) \times X \times Y \times \R \rightarrow \mathbb{R} \,.
\end{equation}
For example (see \Cref{section_4}), it may take the following form 
\[
u(\alpha_t) = \theta_1 \log (\alpha_t + \theta_2) + \theta_3  -   c {\alpha_{t}}  \,,
\]
which is the node's revenue minus the total expenditure of the hardware and electricity cost.
\item The dynamics of states evolve in a controlled Markovian fashion:
the {\bf wealth} state $x$ changes due to mining rewards and the expense of mining, i.e.,
\begin{equation}
\label{wealth_process_for_individual_node}
\diff x_t = \Big( r x_t - c(\alpha_{t})  \Big) \diff t +{k_t b_t}   \diff N_t \,,
\end{equation}
where $c({\alpha_t})>0$ represents the cost
to compute hashes in new generations of a mining machine in terms of dollars.
The {\bf mining reward for a node} depends on
 the total rewards given to the network of nodes;
as well as
the node's probability of receiving its mining reward $\frac{\alpha_{t}}{h_t}$~, which is proportional to the ratio of the hashrate it invests in as a fraction of the total network hashrate.
Thus, it can be modeled by a non-homogenous compound Poisson process $N_t \sim \text{Poisson}({\frac{\lambda_{t}}{h_t}\alpha_t})$,
such that  
\begin{itemize}

\item the {\bf rate at which a network of nodes finds blocks} ${{\lambda}}(t)$ is given in \autoref{expected_block_arrival_rate};

\item the {\bf  number of tokens in each block} $k_t$ is given in \autoref{number_of_block_reward}.
\item the {\bf value of each token} $b_t$ is modeled as an OU-process in Assumption \autoref{assumption_token_price}
\begin{equation}
\label{reward_process_for_individual_node}
\diff b_t =  (\hat b_t - b_t) \diff t +\sigma \diff W_t \,,
\end{equation}
where $W_{t}$ is the standard Wiener process, and the cost of a token production 
$
\hat b_{t}= \beta \frac{c(h_t)}{K_t} 
$.

\end{itemize}

\item 
In contrast to the mean field game model in the current literature \citep{li2019mean}, 
our model has nodes whose states dynamics are coupled 
through the mean optimal control $\bar \alpha_{t}$ in the total hashrate term 
\begin{equation}\label{total_hashrate}
{h}_t := {M(t)} \bar \alpha_t
\qquad \text{where} \qquad
{\bar \alpha_t} = \int_{X} \int_{B} \hat \alpha(t, x, b) m(t, x, b) \diff x \diff b \,.
\end{equation}
Here, ${m(t, x, b)}$ denotes the wealth-token price distribution when every node uses the optimal control $\hat \alpha(t, x, b)$ at time $t$. Note that total hashrate term has appeared in \autoref{token_price} and \autoref{expected_block_arrival_rate}.

\end{enumerate}
\end{game}

\subsection{Infinitesimal Generator}

\begin{lemma}\label{lemma_IVP}
For a fixed $\alpha_t$,
there exists a unique, right-continuous and left-limit F-adapted solution to the 
stochastic differential equation for the jump-diffusion states process
\begin{equation}\label{system}
\diff \left[ \begin{array}{c} x_t \\ b_t  \end{array} \right] 
= \left[ \begin{array}{c} r x_t - c(\alpha_{t}) \\ \hat b_t - b_t  \end{array} \right] \diff t  
+ \left[ \begin{array}{c} {k_t b_t}   \\ 0  \end{array} \right]  \diff N_{t} 
+ \left[ \begin{array}{c} 0 \\ \sigma  \end{array} \right]  \diff W_{t} \, \quad
\end{equation}
with a random initial value $\big(0, x(0),b(0)\big) = \big(0, x_0, b_0\big) $.
\end{lemma}
\begin{proof}
See Theorem 4 \citep{gaviraghi2017theoretical}.
\end{proof}
Denote $\displaystyle x_{s^{-}}:=\lim_{s' \nearrow s} x_{s'}$, 
under the assumption that the jumps are instantaneous,
the differential Poisson $\diff N_t$ conditions the existence of a block arrival,
which behaves asymptotically for small $\lambda \diff t$ in the interval $(t, t+\diff t]$, i.e.,
\[
\Prob \{\diff N_t =n\}=\left\{\begin{array}{ll}
1-{\frac{\lambda_{t}}{h_t}\alpha_t} \diff t, & n=0 \\
{\frac{\lambda_{t}}{h_t}\alpha_t} \diff t, & n=1 \\
0, & n>1
\end{array}\right\}+\mathrm{O}^{2}\left(\diff t \right)\,.
\]
Thus, the change of a testing function $\upsilon\left(t, x_t, b_t \right)$ of the jump-diffusion states process \ref{system}, 
$\Delta \upsilon\left(t, x_t, b_t \right)$, can be decomposed into the sum of continuous and discontinuous changes,
which gives us
{\scriptsize
\begin{eqnarray*}
\Delta \upsilon\left(t, x_t, b_t \right) 
&=&
 \upsilon\left(t+\Delta t, x_{t+\Delta t}, b_{t+\Delta t} \right)   -   \upsilon\left(t, x_t, b_t \right) \\
&=&
 \left[{\frac{\partial \upsilon\left(t, x_t, b_t \right)}{\partial t}}
+\Big( r x_t - c(\alpha_{t})  \Big) \frac{\partial }{\partial x_t} \upsilon\left(t, x_t, b_t \right)    
+(\hat b_t - b_t) \frac{\partial }{\partial b_t} \upsilon\left(t, x_t, b_t \right) + \frac{\sigma^2}{2}  \frac{\partial^2 }{\partial b_t ^2} \upsilon\left(t, x_t, b_t \right) 
\right]  \Delta t \\
&+&  \sigma \frac{\partial }{\partial b_t } \upsilon\left(t, x_t, b_t \right) \Delta W_t 
+ {\frac{\lambda_{t}}{h_t}\alpha (t, x_t, b_t) }\int_{\mathbb{R}}\Big( \upsilon\left(t, x_{t^{-}}+z , b_t\right)-\upsilon\left(t, x_{t^{-}}, b_t\right) \Big) \delta(z-  {k_t b_t} ) \diff z \Delta t
\,.
\end{eqnarray*}
}Taking the expectation $\E[\cdot]$ over wealth and token price,
dividing them by $\Delta t$ and taking the limit $\Delta t \to 0$,
we obtain a Cauchy problem
{\scriptsize
\begin{eqnarray}\label{the Cauchy problem}
 \frac{\diff}{\diff t} \E \left[\upsilon\left(t, x_t, b_t \right) \right] &=& 
\lim_{\Delta t \to 0}\frac{ 
\E \left[ \upsilon\left(t+\Delta t, x_{t+\Delta t}, b_{t+\Delta t} \right) \right]  - \E \left[ \upsilon\left(t, x_t, b_t \right) \right] 
 }{\Delta t} \nonumber  \\
&=&   \left[ {\frac{\partial \upsilon\left(t, x_t, b_t \right)}{\partial t}} 
+\Big( r x_t - c(\alpha_{t})  \Big) \frac{\partial }{\partial x_t} \upsilon\left(t, x_t, b_t \right)    
+(\hat b_t - b_t) \frac{\partial }{\partial b_t} \upsilon\left(t, x_t, b_t \right) + \frac{\sigma^2}{2}  \frac{\partial^2 }{\partial b_t ^2} \upsilon\left(t, x_t, b_t \right) \right]
\nonumber  \\
&& +{\frac{\lambda_{t}}{h_t}\alpha (t, x_t, b_t) }\Big( \int_{\mathbb{R}} \upsilon\left(t, x_{t^{-}}+z, b_{t} \right) \delta(z- {k_t b_t} ) \diff z  -\upsilon\left(t, x_{t^{-}}, b_{t}\right)  \Big) 
\end{eqnarray}
}with an initial condition 
$\upsilon(0,x(0),b(0))=\upsilon(0, x_0, b_0)$.

Let the evolution operators $\left\{\mathcal{T}_{t}\right\}_{t \in [0,\infty)}$ denote the family of flows associated with stochastic process \ref{system}.
By the existence and uniqueness of the solution to the initial-value problem \ref{system}, as demonstrated in Lemma \autoref{lemma_IVP},
the family of flows $\left\{\mathcal{T}_{t}\right\}_{t \in [0,\infty)}$ constitutes semi-group properties \citep{kunita1997stochastic},
allowing us to define the infinitesimal generator $\mathcal{A}$ of the semigroup $\left\{\mathcal{T}_{t}\right\}_{t \in [0,\infty)}$ as 
\begin{equation}\label{operator_A_def}
\mathcal{A} \upsilon := \lim _{t \rightarrow 0} \frac{\mathcal{T}_{t} \upsilon -\upsilon }{t} \,.
\end{equation}
Thus, the infinitesimal generator $\mathcal{A}$ of the flows defined in \autoref{operator_A_def} takes the following form:
\begin{eqnarray}\label{generator_L}
\mathcal{A} \upsilon \left(t,  x, b\right) 
&:= & 
\Big( r x - c(\alpha_{t})  \Big) \frac{\partial }{\partial x} \upsilon\left(t, x, b \right)    
+(\hat b_t - b) \frac{\partial }{\partial b} \upsilon\left(t, x, b \right) + \frac{\sigma^2}{2}  \frac{\partial^2 }{\partial b ^2} \upsilon\left(t, x, b \right)
 \nonumber \\
&&
+{\frac{\lambda_{t}}{h_t}\alpha (t, x, b) } \left(\int_{\mathbb{R}} \upsilon\left(t, x, z \right) \delta(z-x - {k_t b}) \diff z  -\upsilon\left(t,  x, b\right)  \right)
\end{eqnarray}
Let $m(t, x_t, b_t)$ be the density function at time $t$, and denote with $\langle\cdot, \cdot\rangle \text { the } L^{2} $-inner product.
The adjoint operator $\mathcal{A^*} $ is defined as 
\begin{equation}\label{adjoint_def}
\langle \mathcal{A} \upsilon, m \rangle=\left\langle \upsilon, \mathcal{A}^{*} m \right\rangle 
\end{equation}
The left hand side of \autoref{the Cauchy problem} gives
\begin{equation}\label{the Cauchy problem LHS}
\frac{\diff}{\diff t} \mathbb{E}\Big[\upsilon\left(t, x_t, b_t \right)\Big] 
=\int_{X} \int_{B}  \frac{\partial \upsilon(t, x_t, b_t)}{\partial t}  m(t, x_t, b_t) \diff b \diff x 
~+ \int_{X} \int_{B}   \frac{\partial m(t, x_t, b_t)}{\partial t}  \upsilon(t, x_t, b_t) \diff b \diff x \,,
\end{equation}
while the right hand side of \autoref{the Cauchy problem} is
{\scriptsize
\begin{eqnarray}\label{the Cauchy problem RHS}
 \frac{\diff}{\diff t} \E \left[\upsilon\left(t, x_t, b_t \right) \right] &=& 
 \int_{X} \int_{B} m(t, x_t, b_t) \left[ {\frac{\partial \upsilon\left(t, x_t, b_t \right)}{\partial t}} 
+\Big( r x_t - c(\alpha_{t})  \Big) \frac{\partial }{\partial x_t} \upsilon\left(t, x_t, b_t \right)    
+(\hat b_t - b_t) \frac{\partial }{\partial b_t} \upsilon\left(t, x_t, b_t \right) 
\right.
\nonumber  \\
&+& \left.
 \frac{\sigma^2}{2}  \frac{\partial^2 }{\partial b_t ^2} \upsilon\left(t, x_t, b_t \right) 
+{\frac{\lambda_{t}}{h_t}\alpha (t, x_t, b_t) }\Big( \int_{\mathbb{R}} \upsilon\left(t, x_{t^{-}}+z, b_{t} \right) \delta(z- {k_t b_t} ) \diff z  -\upsilon\left(t, x_{t^{-}}, b_{t}\right)  \Big) 
\right]
\diff b \diff x \,.
\nonumber  \\
\end{eqnarray}
}Combining the results of Equations \ref{the Cauchy problem LHS} and \ref{the Cauchy problem RHS}, we obtain
{\scriptsize
\begin{eqnarray}\label{the Cauchy problem new}
 \int_{X} \int_{B}   \frac{\partial m(t, x_t, b_t)}{\partial t}  \upsilon(t, x_t, b_t) \diff b \diff x
= 
 \displaystyle  \int_{X} \int_{B}  m(t, x_t, b_t)   \left(
  \Big( r x_t - c(\alpha_{t})  \Big) \frac{\partial }{\partial x_t} \upsilon\left(t, x_t, b_t \right)    
+(\hat b_t - b_t) \frac{\partial }{\partial b_t} \upsilon\left(t, x_t, b_t \right) 
   \right. \nonumber \\
 \left. 
 + \frac{\sigma^2}{2}  \frac{\partial^2 }{\partial b_t ^2} \upsilon\left(t, x_t, b_t \right)
 +  {\frac{\lambda_{t}}{h_t}\alpha (t, x_t, b_t) }\left(\int_{\mathbb{R}} \upsilon\left(t, x_{t^-}+z, b_t \right) \delta(z- {k_t b_t} ) \diff z  -\upsilon\left(t, x_{t^-}, b_t\right)  \right)
\right) \diff b \diff x    \,.
\end{eqnarray}
}When applying integration by parts to \autoref{the Cauchy problem new} \citep{gaviraghi2017theoretical},
$\mathcal{A^*} $ 
takes the following form:
\begin{eqnarray}\label{generator_L_star}
\mathcal{A^*} m \left(t,  x, b\right) 
&:=& 
 - \frac{\partial }{\partial x} \Big[\Big( r x - c(\alpha_{t})  \Big)  m\left(t,  x, b \right) \Big]
- \frac{\partial }{\partial b} \Big[ \Big(\hat b_t - b\Big)  m\left(t,  x, b \right) \Big] 
+ \frac{\sigma^2}{2}  \frac{\partial^2 }{\partial b ^2}  m \left(t, x, b \right)
 \nonumber \\
&&
\int_{\mathbb{R}}  {\frac{\lambda_{t}}{h_t}\alpha (t, x, z) } m\left(t, x, z\right)  \delta(z - x + {k_t b}) \diff z - {\frac{\lambda_{t}}{h_t}\alpha (t, x, b) }m\left(t,  x, b \right) 
\,.
\end{eqnarray}
The decision of nodes and the evolution of the joint distribution of their wealth-token price process can be summarized using infinitesimal generators. 
Thus, we can use these generators to derive a Hamilton-Jacobi-Bellman equation and a Fokker-Planck equation.

\subsection{Hamilton-Jacobi-Bellman Equation}

In this subsection, we aim to define an operator to capture the phenomenon of each node solving stochastic control problem \ref{expected_total_utility}, subject to the underlying controlled jump-diffusion dynamics \ref{system} for a given population distribution.

We can achieve this by applying Bellman's principle of optimality and then working out the optimal control \ref{control} backwards in time.
To do so, define the value function
\begin{equation}\label{value_function}
v(0, x_0, b_0) := \sup_{\alpha_{t}}\E\left[\int_0^{\infty} e^{- {{r}} t} u(t, x_{t}, b_t, \alpha_{t}) \diff t \, \Big| \, x(0) = x_{0} , b(0)=b_0\right] \,.
\end{equation}
Inspired by the lemma proposed by \cite{li2019mean}, which states that having more wealth gives nodes greater flexibility to choose their hashrates, we have made the following assumption.
\begin{assumption}\label{control_increasing_in_wealth}
Fix a choice of ${\alpha}\in O (t) $,  
the value function $ v(t, x , y) $ is finite and strictly increasing in wealth $x, \forall t \in\left[0, \infty \right)$. 
 \end{assumption} 
The value function defined in \autoref{value_function} satisfies Bellman's principle of optimality
\begin{equation}\label{discrete_time_Bellman_optimality}
v\left(t, x(t), b(t) \right) = \sup _{\alpha_t}\Big\{ u(t, x, b, \alpha )  \Delta t +  e^{-{{r}} \Delta t} ~\E \left[ v\left(t+\Delta t , x_{t+\Delta t}, b_{t+\Delta t}\right) \right]\Big\} \,,
\end{equation}
where $\E[\cdot]$ is the expectation over wealth $x({t+\Delta t})$ and token price $b(t+\Delta t)$.

First, we substitute $e^{-{{r}} \Delta t}  = 1- {{r}} \Delta t +\mathcal{O}(\Delta t^2)$ into  \autoref{discrete_time_Bellman_optimality}
\begin{equation}
v\left(t, x(t), b(t)\right)  = \sup _{\alpha_t}\Big\{ u(t, x, b,\alpha )  \Delta t +  (1-{{r}} \Delta t) ~\E \left[ v\left(t + \Delta t, x_{t+\Delta t}, b_{t+\Delta t}\right) \right] \Big\} + \mathcal{O}(\Delta t^2) \,,
\end{equation}
and subtract $(1- {{r}} \Delta t)v(t, x(t), b(t))$ from both sides to obtain
{\scriptsize
\begin{equation}\label{discrete_time_Bellman_optimality_1}
{{r}} v \left(t, x(t), b(t)\right) \Delta t  = \sup _{\alpha_t}\Big\{ u(t, x, b,\alpha )  \Delta t +  (1-{{r}} \Delta t) ~\E \big[v\left(t + \Delta t, x_{t+\Delta t}, b_{t+\Delta t}\right) - v\left(t, x(t), b(t)\right)\big] \Big\} + \mathcal{O}(\Delta t^2)
\end{equation}}Then, dividing the above expression by $\Delta t$, we find
{\footnotesize
\begin{equation}\label{discrete_time_Bellman_optimality_2}
{{r}} v \left(t, x(t), b(t)\right)  = \sup _{\alpha_t}\Big\{ u(t, x, b,\alpha ) +  (1-{{r}} \Delta t) ~\E \Big[\frac{ v\left(t + \Delta t, x_{t+\Delta t}, b_{t+\Delta t}\right) -  v\left(t, x(t), b(t)\right)}{ \Delta t}\Big] \Big\} + \mathcal{O}(\Delta t)
\end{equation}}Taking the limit $\Delta t \rightarrow 0$,
\autoref{discrete_time_Bellman_optimality_2} yields
\begin{equation}\label{discrete_time_Bellman_optimality_3}
{{r}} v \left(t, x(t), b(t)\right)  = \sup _{\alpha_t}\Big\{ u(t, x, b,\alpha )  +  \frac{\diff}{\diff t}  \E \left[  v\left(t, x(t), b(t)\right)\right] \Big\} 
\end{equation}
Using \autoref{the Cauchy problem}, 
we obtain the Hamilton-Jacobi-Bellman equation:
\begin{eqnarray}\label{HJB}
{{r}} v \left(t, x, b \right) &= &\sup_{\alpha} \left\{ u(t, x, b,\alpha )
+\Big( r x - c(\alpha_{t})  \Big) \frac{\partial }{\partial x} v\left(t, x, b \right) \right.
\nonumber \\
&&~~~~~~~~\left. +\frac{\lambda_{t}}{h_t}\alpha (t, x, b) \left(\int_{\mathbb{R}} v\left(t, x, z \right) \delta(z-x - {k_t b}) \diff z  -v\left(t, x, b\right)  \right)\right\}
\nonumber \\
 &&~~~~~~~~+(\hat b_t - b) \frac{\partial }{\partial b} v\left(t, x, b \right) + \frac{\sigma^2}{2}  \frac{\partial^2 }{\partial b ^2} v\left(t, x, b \right)
+{\frac{\partial v\left(t, x, b \right)}{\partial t}} 
\end{eqnarray}
The optimal control $\hat \alpha_t$ maximizes the following objective:
\begin{equation}\label{optimal_control_maximizes_following_objective}
u(t, x, b,\alpha_t )
+\Big( r x - c(\alpha_{t})  \Big) \frac{\partial }{\partial x} v\left(t, x, b \right) 
+ \frac{\lambda_{t}}{h_t}\alpha_t \int_{\mathbb{R}}\Big( v\left(t, x, z\right)-v\left(t, x, b\right) \Big) \delta(z-x - {k_t b}) \diff z \,.
\end{equation}
Thus, \autoref{optimal_control_maximizes_following_objective} can be used to recover optimal control.
Since the node's expenditure has to be less than or equal to its wealth, we can refine the set of admissible Markovian feedback control in \autoref{control} as 
\begin{eqnarray}\label{outcome}
O(t) := \left\{ \hat \alpha_{t} \geq 0 ~\Big| ~
\hat \alpha_t =
\arg\max \Big\{ u(t, x, b,\alpha_t )
+\Big( r x - c(\alpha_{t})  \Big) \frac{\partial }{\partial x} v\left(t, x, b \right) \qquad\qquad\qquad\quad
\right. \nonumber  \\
+ \left. \frac{\lambda_{t}}{h_t}\alpha_t \int_{\mathbb{R}}\Big( v\left(t, x, z \right)\delta(z-x - {k_t b}) \diff z -v\left(t, x, b\right) \Big) \Big\} \right\} \,.
\end{eqnarray}
In particular, the active nodes comprise the set 
\begin{equation}\label{optimal_control}
O^*(t) := \left\{ \hat \alpha_{t} >0 ~\Big| ~
\hat \alpha_t \in O(t) \right\} \,.
\end{equation}
\begin{assumption}[Boundary Conditions]
Assuming that the token price process is reflected at the line $b = 0$ and $b = \max (Y)$,
one can show that this gives rise to the following boundary conditions for $v$ in $b$-direction:
\begin{equation}\label{v_b_BC}
\frac{\partial v(t, x, b)}{\partial {b}} \Big |_{b = 0}=\frac{\partial v(t, x, b)}{\partial {b}} \Big |_{b = \max (Y)} = 0 \,.
\end{equation}
Let $\max (X)$ be large enough to obtain
\begin{equation}\label{v_xmax_BC}
\frac{\partial v(x,b)}{\partial {x}} \Big |_{x = \max (X)}  = 0 \,.
\end{equation}
Since nodes cease their mining if their wealth
hits $0$,
so that the boundary condition with respect to the wealth dimension at $x=0$ is chosen to be
\begin{equation}\label{x_BC}
c({\alpha_{t}} ) |_{x=0}= 0 \,.
\end{equation}
\end{assumption}
\begin{remark}
The viscosity solutions of Hamilton-Jacobi equations with such boundary conditions have been studied in~\cite{qiu2018viscosity}. 
Here, it is convenient to introduce the Hamiltonian for \autoref{HJB}
{\scriptsize
\begin{equation}
H(x, b,p) := \sup_{\alpha} \left\{  u(t, x, b,\alpha ) + \Big( r x  - c(\alpha_{t}) \Big) p
+ \frac{\lambda_{t}}{h_t}\alpha_t \int_{\mathbb{R}}\Big( v\left(t, x, z \right)\delta(z-x - {k_t b}) \diff z -v\left(t, x, b\right) \Big) \Big\} \right\} \,.
\end{equation}
}The non-decreasing and non-increasing envelopes $H^{\uparrow}$ and $H^{\downarrow}$ of $p \mapsto H(x, b,p)$ can be defined as
{\scriptsize
\begin{equation*}
  \begin{split}
H^\uparrow (x, b,p)&=  \sup_{0\leq c(\alpha)\leq rx} \left\{  u(t, x, b,\alpha ) + \Big( r x  - c(\alpha_{t}) \Big) p
+ \frac{\lambda_{t}}{h_t}\alpha_t \int_{\mathbb{R}}\Big( v\left(t, x, z \right)\delta(z-x - {k_t b}) \diff z -v\left(t, x, b\right) \Big) \Big\} \right\} \, ,\\
H^\downarrow (x, b,p)&=  \sup_{\max(0, rx)\leq c(\alpha)} \left\{  u(t, x, b,\alpha ) + \Big( r x  - c(\alpha_{t}) \Big) p
+ \frac{\lambda_{t}}{h_t}\alpha_t \int_{\mathbb{R}}\Big( v\left(t, x, z \right)\delta(z-x - {k_t b}) \diff z -v\left(t, x, b\right) \Big) \Big\} \right\} \,. 
  \end{split}
\end{equation*}
}Similar to the method as presented in \cite{achdou2020mean}, it can be seen that 
\begin{displaymath}
  H(x, b,p)=H^\uparrow (x, b,p)+H^\downarrow (x, b,p)- \min_{p\in \R} H(x, b,p).
\end{displaymath}
 Thus, the {boundary condition} for the value function associated with the state constraint satisfies
{\scriptsize
\begin{equation}\label{HJB_BC}
 \hspace{-1cm}
{{r}} v \left(t, 0, b \right) = \sup_{\alpha} \left\{ u(t, 0, b,\alpha )
+ H^{\uparrow}(0,y,\partial_x v |_{x=0}) \right\}
+(\hat b_t - b) \frac{\partial }{\partial b} v\left(t, 0, b \right) + \frac{\sigma^2}{2}  \frac{\partial^2 }{\partial b ^2} v\left(t, 0, b \right)
+{\frac{\partial v\left(t, 0, b \right)}{\partial t}} \,.
\end{equation}
}\vspace{-0.5cm}
\end{remark}
Overall, we can interpret the properties of the solution to the Hamilton-Jacobi-Bellman \autoref{HJB} 
using the operator $\mathfrak{A}$.

\begin{definition}
\label{optimal_control_operator}

Define an optimal control operator 
\[
\mathfrak{A}: {{[0,\infty) \times X \times Y}} \rightarrow [0,\infty) \times O,
\] 
 where
{\footnotesize
\[
 \mathfrak{A}[m](t,x, b) = \Big\{ \hat \alpha(t,x, b)  \in O ~\Big|~ \text{ value function } v \text{ satisfies \autoref{HJB} and \autoref{HJB_BC} for a given distribution } m \Big\},
\]}which captures the postulate that nodes will optimize their running payoffs when the distribution of other nodes is fixed.
\end{definition}

\subsection{Fokker-Planck Equation}

Population behavior can be described through evolutionary distribution of nodes. In this subsection, we aim to define an operator to capture population behavior arising from nodes under optimal control.

First, using \autoref{generator_L_star} and letting $ {\frac{\partial m\left(t, x, b \right)}{\partial t}} = \mathcal{A^*} m \left(t, x, b\right) $, we can get a Fokker-Planck equation, i.e.,
{\footnotesize
\begin{eqnarray}\label{FP}
 {\frac{\partial m\left(t, x, b \right)}{\partial t}} 
&=& - \frac{\partial }{\partial x} \Big[\Big( r x - c(\alpha_{t})  \Big)  m\left(t, x, b \right) \Big]
- \frac{\partial }{\partial b} \Big[ \Big(\hat b_t - b\Big)  m\left(t, x, b \right) \Big] 
+ \frac{\sigma^2}{2}  \frac{\partial^2 }{\partial b ^2}  m \left(t, x, b \right)
 \nonumber \\
&& 
+\int_{\mathbb{R}} \frac{\lambda_{t}}{h_t}\alpha (t, x, z) m\left(t, x, z\right)  \delta(z -x + {k_t b}) \diff z - \frac{\lambda_{t}}{h_t}\alpha (t, x, b)m\left(t, x, b \right) \,.
\end{eqnarray}
}with initial condition
$
m \big(0, x(0),b(0)\big) = m_0 (x_0, b_0) ;
$.
\vspace{0.2cm}
\begin{assumption}[Boundary Conditions]
For $0 \leq x \leq {k_t b}$,
there is no density at $x - {k_t b}$ jumping to $x$.
Thus, density $m\left(t, x, b \right)$ satisfies the following boundary conditions
{\footnotesize
\begin{equation}\label{FP_BC}
 {\frac{\partial m\left(t, x, b \right)}{\partial t}} 
= - \frac{\partial }{\partial x} \Big[\Big( r x - c(\alpha_{t})  \Big)  m\left(t, x, b \right) \Big]
- \frac{\partial }{\partial b} \Big[ \Big(\hat b_t - b\Big)  m\left(t, x, b \right) \Big] 
+ \frac{\sigma^2}{2}  \frac{\partial^2 }{\partial b ^2}  m \left(t, x, b \right)
 - \frac{\lambda_{t}}{h_t}\alpha (t, x, b)m\left(t, x, b \right) \,.
\end{equation}
 }
\end{assumption}
We solve for $m\left(t, x, b \right)$, where the sum of a measure is continuous with respect to a two-dimensional Lebesgue measure on $(0,\infty)\times [0,\infty) $ with a density $\bar m$  and a measure $\eta$  that is supported by the line $\{x=0\} \times  [0, \infty)$:
 \begin{equation}
   \label{eq:macroeco:22}
    \diff m(t, x, b) =  {\bar m}(t, x, b)  \diff x \diff b +  \diff \eta(t,b) \, . \vspace{0.2cm}
 \end{equation}
Thus, \autoref{FP} has a weak form for every test function $\upsilon$:
{\scriptsize
  \begin{equation}
\label{FP_BC}
 \left. \begin{array}[c]{r}
    \displaystyle 
 \int_{x> 0} \int_{b \geq 0} {\bar m}(t,x, b) \left( H_p \Big(x ,b,\partial_x v (t,x, b) \Big)  \partial_x \upsilon(t,x, b)
 +(\hat b_t - b) \frac{\partial }{\partial b} \upsilon\left(t, 0, b \right) + \frac{\sigma^2}{2}  \frac{\partial^2 }{\partial b ^2} \upsilon\left(t, 0, b \right)
   \right)\diff x \diff b 
\\
 + \displaystyle
 \int_{b \geq 0}  \left(H_p^\uparrow \Big(0 ,b,\partial_x v(t, 0_+,b) \Big)  \partial_x \upsilon (t, 0,b) 
 +(\hat b_t - b) \frac{\partial }{\partial b} \upsilon\left(t, 0, b \right) + \frac{\sigma^2}{2}  \frac{\partial^2 }{\partial b ^2} \upsilon\left(t, 0, b \right)
   \right) \diff\eta(b)
  \end{array}
\right\}=0
  \end{equation}
  }

Hence, we can define the operator $\mathfrak{M}$ in a way that corresponds to the solution to the Fokker-Planck  \autoref{FP} and captures the phenomenon that a new distribution arises every time nodes use an optimal strategy.
\begin{definition}\label{distribution_evolution_operator}
Define a distribution evolution operator 
\[
\mathfrak{M} : [0,\infty) \times {O} \times {{X \times Y}} \rightarrow [0,\infty) \times {{X \times Y}}
\] 
where
{
\footnotesize
\[
 \mathfrak{M}[\hat \alpha](t,x, b) = \Big\{ m(t,x, b) ~\Big|~ \text{distribution }  m \text{ satisfies \autoref{FP} and \autoref{FP_BC} for given optimal control } \hat \alpha \Big\},
\]}maps a mean field control $\hat \alpha(t, \cdot,\cdot)$ and the distribution of nodes' states $m(t,\cdot,\cdot)$ to a new distribution
$\mathfrak{M}\left(\hat \alpha(t, \cdot,\cdot), m(t,\cdot,\cdot) \right)$.

\end{definition}

\subsection{Mean Field Equilibrium}\label{Mean Field Equilibrium Def}
\n 
The composition of the two operators described in Definition  \ref{optimal_control_operator} and Definition \ref{distribution_evolution_operator} leads to a fixed point,  giving the MFE of the game. 
Because optimal control gives rise to dynamics involving a large number of nodes, a single node has a negligible effect on the game's outcome as the number of nodes increases: the effect of other nodes on a single node's payoff and fluctuations in motion dynamics will  ``average out".  Thus, we expect that states remain roughly constant over time. To formalize this notion, we define MFE as follows:
\begin{definition}[Mean Field Equilibrium]\label{MEGF_def}
The mean hashrate is called a Mean Field Equilibrium (MFE) of mean field game \ref{general_structure}, 
if and only if
\[
{\bar \alpha_t} = \int_{X} \int_{B} \hat \alpha(t, x, b) m(t, x, b) \diff x \diff b \,,
\]
where
$\hat \alpha(t, \cdot,\cdot) \in \mathfrak{A}\left(m(t, \cdot,\cdot)\right)$, 
 and 
the invariant distribution 
$m(t,\cdot,\cdot) \in\mathfrak{M}\left(\hat \alpha(t, \cdot,\cdot), m(t,\cdot,\cdot) \right)$
is a fixed point of the operator 
\[
\Phi : {[0,\infty) \times {X \times Y}} \to {[0,\infty) \times {X \times Y}},
\] 
such that
\[
\Phi \Big(m(t,\cdot,\cdot) \Big) = \mathfrak{M} \Big(\mathfrak{A}\big(m(t,\cdot,\cdot) \big), m(t,\cdot,\cdot) \Big) \,.
\]

\end{definition}

The MFE as described in Definition \autoref{MEGF_def} is the solution to the time-dependent backward-forward system PDEs 
\ref{HJB} and \ref{FP}, with boundary conditions \ref{HJB_BC} and \ref{FP_BC}
and with parameters (\ref{expected_block_arrival_rate}--\ref{token_price},
\ref{expected_block_arrival_rate}--\ref{total_hashrate}),
where the density $m(t, x, b)$ satisfies the initial condition
$
m \big(0, x(0),b(0)\big) = m_0 (x_0, b_0) ;
$
and 
the value function $v(t, x, b)$ satisfies a terminal condition
$$
v(\infty, x, b)=v_{\infty}(x, b).
$$
The value function $v_{\infty}(x, b)$ is the solution to the time-independent version of these system PDEs \ref{HJB} and \ref{FP}, 
when the state dynamics induced by $\hat \alpha({\infty}, \cdot,\cdot)$ and $m({\infty},\cdot,\cdot)$ 
form an invariant distributions over the infinite-time horizon.
\begin{remark}
As $t \to \infty$, 
the steady state will be reached when 
${\frac{\partial v\left(t, x, b \right)}{\partial t}} \to 0$
and
$~{\frac{\partial m\left(t, x, b \right)}{\partial t}} \to 0$.
We denote the number of cumulated tokens \ref{number_of_cumulated_tokens} to be $K_{{\infty}}$, the number of block rewards \ref{number_of_block_reward} to be $k_{{\infty}}$, and the block arrival rate \ref{expected_block_arrival_rate} to be
${{\lambda}}_{{\infty}}$, respectively, at the infinite-time horizon $t \to \infty$, i.e.,
\begin{equation}
\footnotesize
K_{{\infty}} = \sum_{l=0}^{\infty} \frac{210000 \times 50}{2^{l}}  \to 2.1 \times 10^7; \quad 
k_{{\infty}} = 50 \left(\frac{1}{2}\right)^{l} + \bar k \to \bar k; \quad  
{{\lambda}}_{{\infty}} =  \frac{1}{600} \left( \frac{{  1- \sqrt[M]{1- \frac{2016}{{H_{s-1}}} } }}{1- \sqrt[M]{1- \frac{2016}{H_s}} }\right)
 \to \frac{1}{600} 
 \end{equation}
 Substituting token price \ref{token_price} and reward arrival intensity \ref{expected_block_arrival_rate}, the system of time-independent PDEs becomes
 {\scriptsize
\begin{equation}\label{system_pde}
 \hspace{-0.2cm}
\begin{cases}
{{r}} v_{{\infty}} \left(x, b \right)
&=~~ \displaystyle \sup_{\alpha} \left\{ u\Big(t, x, b,\alpha_{\infty} (x, b) \Big)
+ \Big( r x - c\big({\alpha_{\infty} (x, b)\big) \Big)}\frac{\partial }{\partial x}  v_{{\infty}} \left(x, b \right)  
+\frac{\lambda_{\infty}}{\bar \alpha_{\infty} M} \alpha_{\infty}(x, b) v_{{\infty}} \left(x + k_{\infty} b, b \right)  \right\}
\vspace{0.2cm} \\
 & \qquad \qquad
- \frac{\lambda_{\infty}}{\bar \alpha_{\infty} M} \alpha_{\infty}(x, b) v_{{\infty}} \left(x, b \right)  +\left(\frac{\beta c \big( \bar \alpha_{\infty} M \big)}{K_{\infty}}  - b \right) \frac{\partial }{\partial b}  v_{{\infty}} \left(x, b \right) + \frac{\sigma^2}{2}  \frac{\partial^2 }{\partial b ^2}  v_{{\infty}} \left(x, b \right)
\vspace{0.2cm} \\
\qquad \hat \alpha_{\infty} &=
\quad \arg\max  \left\{ u\Big(t, x, b,\alpha_{\infty} (x, b) \Big)
+ \Big( r x - c\big({\alpha_{\infty} (x, b)\big) \Big)}\frac{\partial }{\partial x}  v_{{\infty}} \left(x, b \right)  
+\frac{\lambda_{\infty}}{\bar \alpha_{\infty} M} \alpha_{\infty}(x, b) v_{{\infty}} \left(x + k_{\infty} b, b \right)  \right\}
 \vspace{0.2cm}  \\
\qquad 0
&=~~
- \frac{\partial }{\partial x} \Big[ \Big(r x - c\big({\hat \alpha_{\infty}  (x, b) \big) \Big)}  m_{{\infty}}\left(x, b \right)  \Big]
- \frac{\partial }{\partial b} \Big[ \Big( \frac{\beta c \big( \bar \alpha_{\infty} M \big)}{K_{\infty}} - b\Big)  m_{{\infty}}\left(x, b \right)  \Big] 
+ \frac{\sigma^2}{2}  \frac{\partial^2 }{\partial b ^2}  m_{{\infty}}\left(x, b \right) 
\vspace{0.2cm}  \\
 & \qquad \qquad + \frac{\lambda_{\infty}}{\bar \alpha_{\infty} M} \hat \alpha_{\infty} (x- k_{\infty} b, b)  m_{{\infty}}\left(x - k_{\infty} b , b \right) - \frac{\lambda_{\infty}}{\bar \alpha_{\infty} M} \hat \alpha_{\infty} (x, b)m_{{\infty}}\left(x, b \right) 
\vspace{0.2cm}   \\
\qquad \bar \alpha_{\infty} &=~~ \displaystyle \int_{X} \int_{B} \hat \alpha_{\infty}(x, b) m_{\infty}(x, b) \diff x \diff b \,.
\end{cases}
\end{equation}}
\end{remark}
The MFE defined above fully characterizes the evolution of the mining game. 
The numerical method for solving the MFE has been well-developed. 
Here, we apply a methodology based on finite difference schemes as presented in \cite{achdou2020mean}.
The steps are organized as follows: 
\begin{description}
\item[Step 1]\label{step_1} Use finite difference schemes to solve 
the system of time-independent PDEs \ref{system_pde} iteratively, for 
a steady terminal condition at the infinite time horizon $t \to \infty$.
\item[Step 2]\label{step_2} Initialize $\bar \alpha(t)$, and
for each time iteration
\begin{enumerate}[label=(\Roman*), font=\upshape] 
\item \label{step_21} Solve for the value function $v(t,x, b)$ in the Hamilton-Jacobi-Bellman \autoref{HJB}, which runs backwards in time with terminal condition $v^{\infty}(x, b)$, as solved in Step 1.
\item \label{step_22} Recover the optimal control $\hat \alpha(t,x, b)$ using \autoref{numerical_control} with the value function $v(t,x, b)$, as solved in Step 2 \ref{step_21}.
\item \label{step_23} Solve for the density function $m(t,x, b)$ in the Fokker-Planck \autoref{FP} using the optimal control $\hat \alpha(t,x, b)$ from Step 2 \ref{step_22}, which runs forward in time with an initial condition $m_0 (x_0, b_0)$. 
\item \label{step_24} For each time iteration, 
introduce a parameter of inertia, $w \in [0, 1)$, to reduce oscillations in searching for the equilibrium (here we set $w :=  \frac{\| \bar \alpha(t)^{\mathrm{new}} - \bar \alpha(t) \|}{\| \bar \alpha(t)^{\mathrm{new}} - \bar \alpha(t) \|_{\infty }} $ for fast convergence);
and update the mean hashrate $\bar \alpha(t)$ according to
\begin{equation}\label{convergence_critiria}
\bar{\alpha}({t})^{\mathrm{new}}:=w \bar{\alpha}({t})+(1-w) \int_{\mathbb{R}} \int_{\mathbb{R}} \hat \alpha(t, x, b) m(t, x, b) \diff x  \diff b \, ,
\end{equation}
using
the optimal control $\hat \alpha(t,x, b)$ from Step 2 \ref{step_22} and the density function $m(t,x, b)$ from Step 2 \ref{step_23}.
\end{enumerate}
\item[Step 3] \label{step_4}
Repeat Step 2 with $\bar{\alpha}(t) = \bar{\alpha}(t)^{\mathrm{new}}$ until convergence results.
\end{description}

In the next section, we will demonstrate a numerical result
 for this MFE
  and test if mappings $\Psi$ and $\vpi$ in the Bitcoin consensus protocol, as described in Problem  \autoref{design_problem}, is well designed.

\section{Results}\label{section_4} 

In this section, we use mappings  $\Psi$ and $\vpi$ in the Bitcoin consensus protocol (as described in \Cref{example_POW}),
as well as estimated parameters in Game \autoref{general_structure},
 to demonstrate a numerical result for the MFE (as described in \Cref{Mean Field Equilibrium Def}). We also analyze mining profitability and blockchain security for this equilibrium, and discuss whether the Bitcoin consensus protocol is well designed.

\subsection{Parameter Estimation}

Let the time to start mining $(t=0)$ be 03-Jan-2009 13:15:00. 
For simplicity, we have altered the timescale from seconds to fortnights when presenting a numerical solution. 
Assume the interest rate is 2\% per annum, which can be rescaled to $r = 7.67 \times 10^{-4}$ per fortnight. 
Set $I=200$ discrete points in the wealth dimension and $J=220$ points in the token price dimension.
Let $\Delta x = 5 \times 10^{13}$ and $\Delta b = 4.6 \times 10^{13}$ denote the equal distance between grid points.

\begin{figure}[h]
\centering
\begin{subfigure}[t]{0.42\textwidth}
\subcaption{
}
\label{fig_technologb_progress}
\centering
{\includegraphics[width=0.88\linewidth]{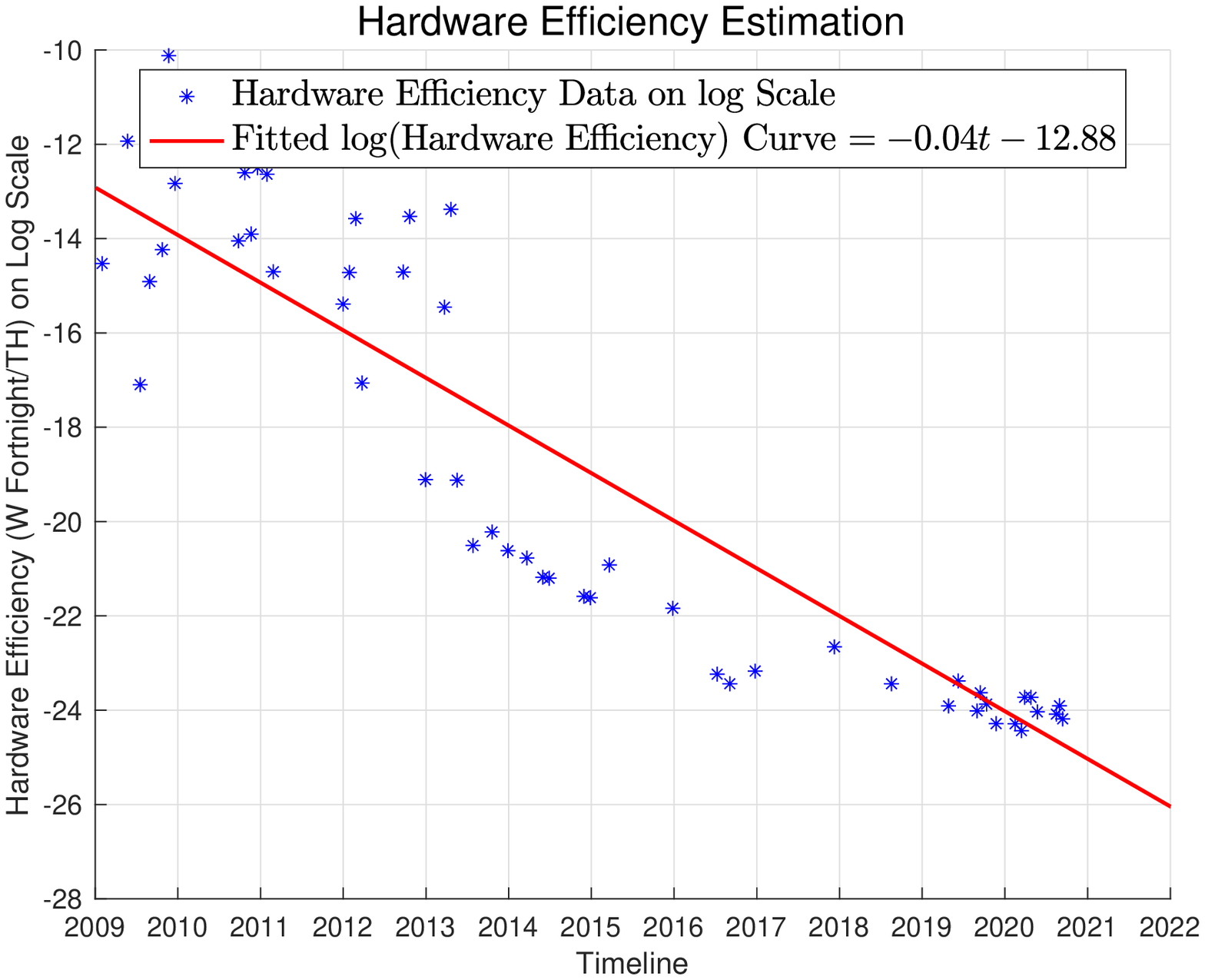}}
\end{subfigure}
\quad
\begin{subfigure}[t]{0.42\textwidth}
\subcaption{
}
\label{fig_revenue}
\centering
{\includegraphics[width=0.88\linewidth]{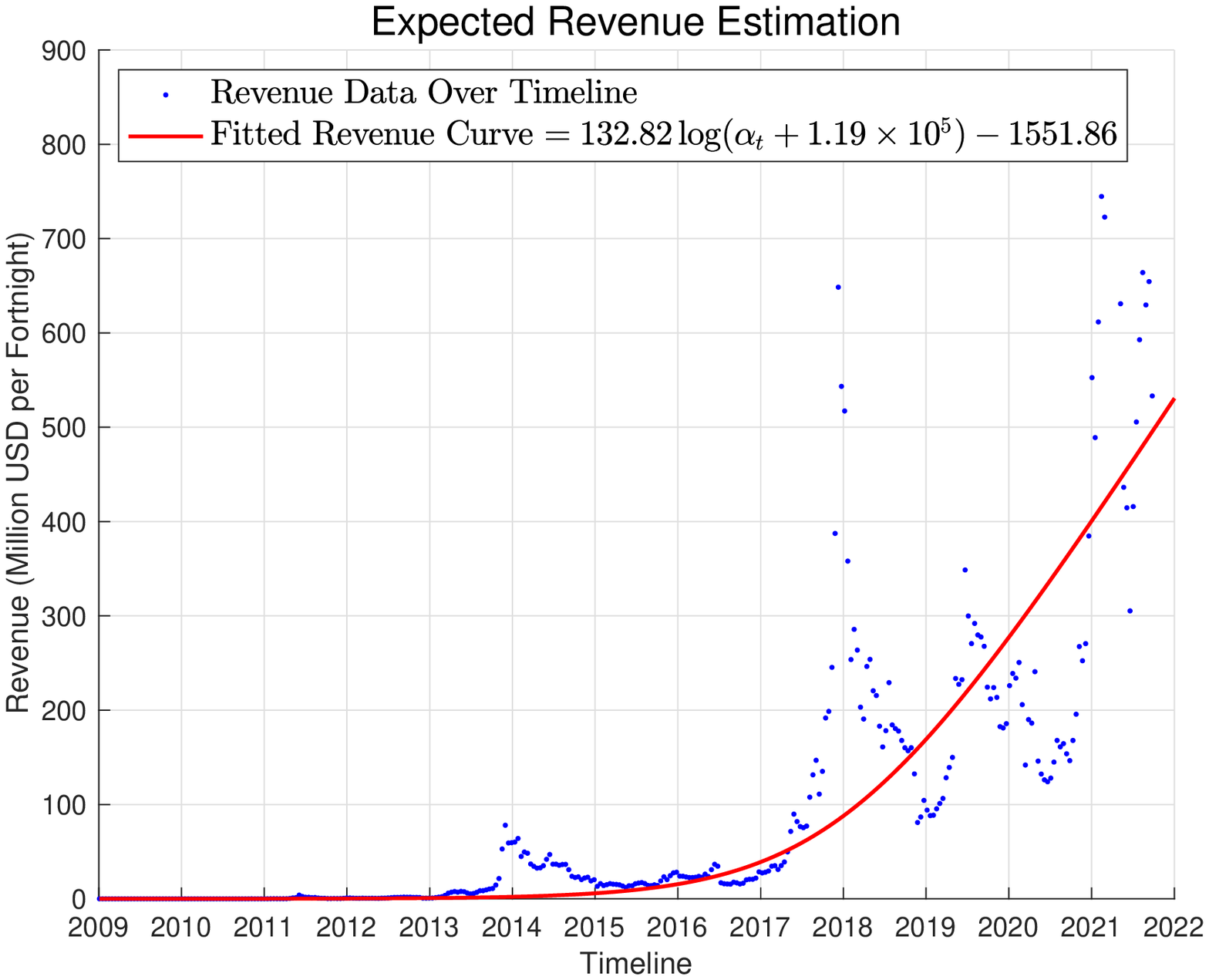}}
\end{subfigure}
\medskip
\vspace{0.3cm}
\begin{subfigure}[t]{0.42\textwidth}
\subcaption{}
\label{fig_costs}
\centering
{\includegraphics[width=0.88\linewidth]{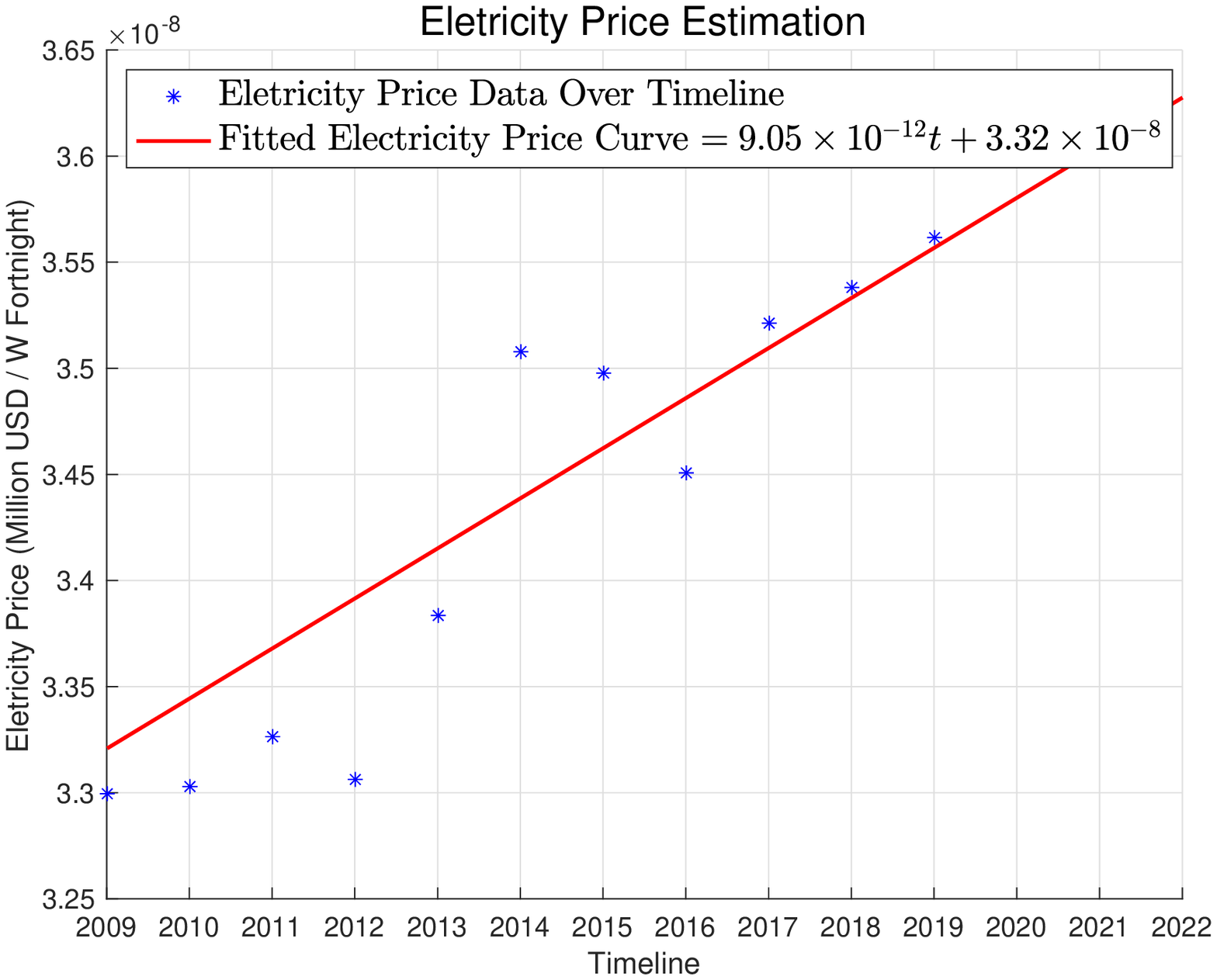}}
\end{subfigure}
\quad
\begin{subfigure}[t]{0.42\textwidth}
\subcaption{}
\label{fig_Number_of_Nodes_Model}
\centering{\includegraphics[width=0.86\linewidth]{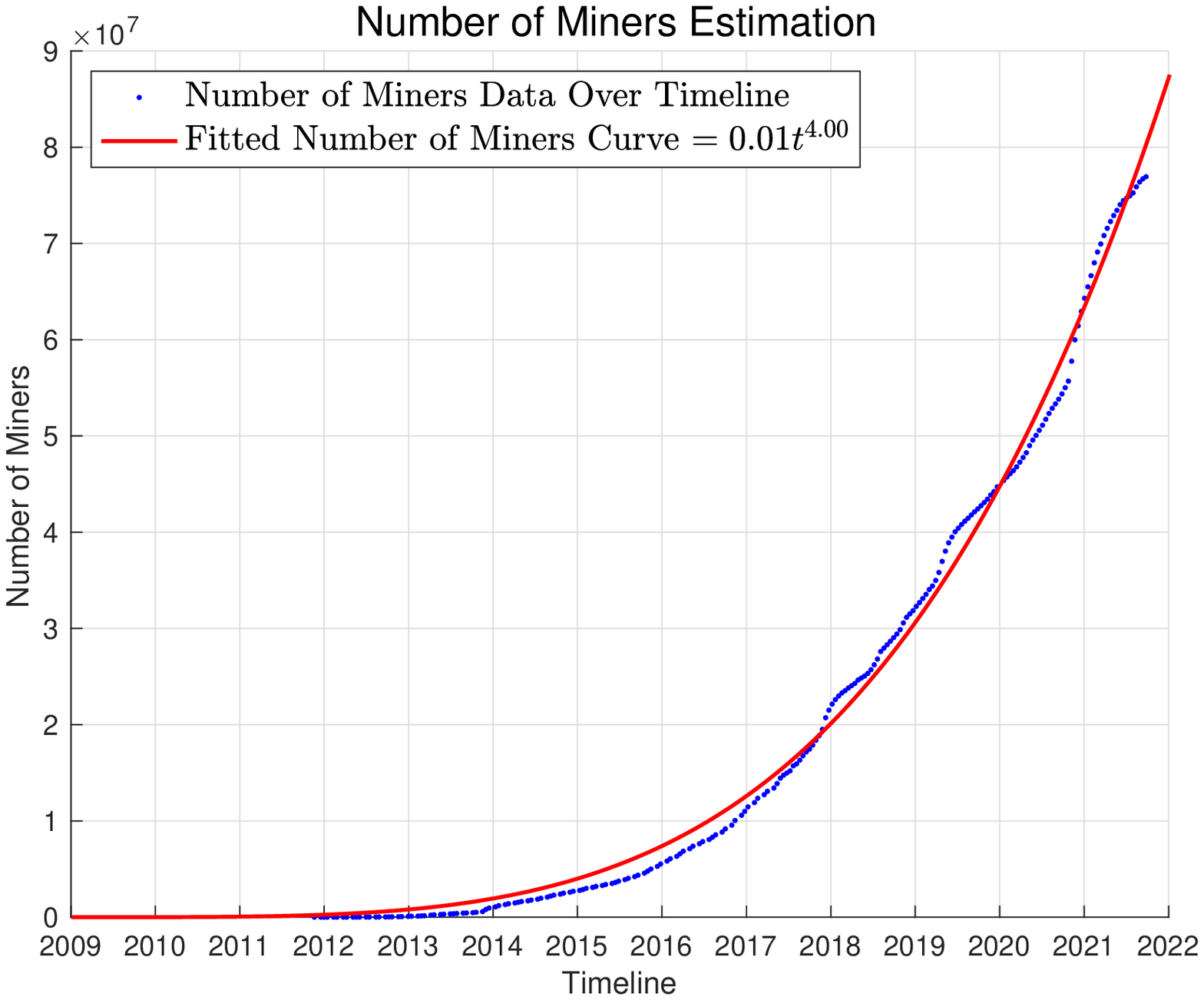}} 
\end{subfigure}
\caption{Fitting curves to estimate model parameters:  
(\subref{fig_technologb_progress}) Fitting a curve to the energy efficiency data of mining hardware over a  timeline, to estimate the average initial efficiency value and the technology discount rate.     
(\subref{fig_revenue}) Fitting a curve to revenue data from different nodes, to estimate how much revenue each unit of TeraHashes can yield over the time horizon. 
(\subref{fig_costs}) Fitting a curve to cost data over the timeline, to estimate the cost of the hashrate when discounted by the rate of technological progress.  
(\subref{fig_Number_of_Nodes_Model}) Fitting a curve to data on node quantity over the timeline, to estimate the growth of number of nodes in the network.
}
\end{figure}

We conducted an extensive and holistic review of the literature, extracting data from peer-reviewed scientific publications \citep{pathirana2019energy, hayes2017cryptocurrency}, and mining machine stores websites, to thoroughly investigating mining hardware efficiencies. 
By assuming that mining hardware converts electricity into hashing computation at the highest possible rate, mining hardware efficiencies can be calculated by taking the ratio of power consumption data to the corresponding hashrate it produces. Given that the resulting ratio decreases exponentially over time, we plotted this efficiency data in a log scale over the timeline, as demonstrated in \autoref{fig_technologb_progress}. 
From the fitted curve, the coefficients are $-0.04  \in [-0.05, -0.03]$ and $-12.88 \in [-14.16, -11.61]$ with $95\%$ confidence bounds.
Thus, the average energy efficiency of mining hardware at the initial time is estimated to be $e^{-12.88} = 2.55 \times 10^{-6}$ (W per TeraHash/Fortnight), and the technology discount rate is estimated to be $0.04$ (per Fortnight).

Assuming that nodes study the relationship between revenue and their hashrate over time, we collected real data \footnote{Special thanks to ListedReserve Pty Ltd for providing a rich data source}  
from 1449 different nodes as well as from reputable commercial sources (such as blockchain.com) to estimate how much revenue each unit of TeraHashes can yield. 
A curve fitting tool was applied as shown in \autoref{fig_revenue}; 
the coefficients are $132.8  \in [102.9, 162.8]$ and $1.19 \times 10^5 \in [4.22 \times 10^4, 1.95 \times 10^5]$ with $95\%$ confidence bounds.  
Our findings suggest that the expected revenue can be expressed in the form of
$\theta_1 \log (\alpha_t + \theta_2) + \theta_3$.

Moreover, by assuming that the primary ongoing cost of mining is driven by the sum of the cost of electricity and the cost of updating mining equipment,
we plotted the ongoing cost over the timeline. 
Notably, 
the growth rate is approximately linear, as depicted in \autoref{fig_costs}.
In the fitted curve, the coefficients are $9.05 \times 10^{-12}  \in [9.04 \times 10^{-12}, 9.05 \times 10^{-12}]$ and $3.32 \times 10^{-8} \in [3.32 \times 10^{-8}, 3.32 \times 10^{-8}]$ with $95\%$ confidence bounds.
This shows that mining software and hardware are exponentially improving at the rate of technological progress over the time horizon. 
The linear cost growth in \autoref{fig_costs} is negligible compared to the exponential growth rate of technological progress. 
The cost of the hashrate discounted at the rate of 
technological progress is the product of the primary ongoing cost $3.32 \times 10^{-8}$ (Million USD / W Fortnight) and the average energy efficiency of mining hardware $2.55 \times 10^{-6}$ (W per TeraHash / Fortnight) at the initial time, which gives $c = 8.43 \times 10^{-14}$ (Million USD per TeraHash).

For simplicity, we have also assumed the cost function to be linear  $c (\alpha) = c \alpha$.
Given the lowest possible price for the energy required and the most efficient equipment (as described above),
the utility (in Million USD) can be written as the difference between revenue and cost, which is represented as
\begin{equation}\label{estimate_utility_function}
u(\alpha_t) = 132.82 \log(\alpha_t+1.19 \times 10^5)  - (8.43 \times 10^{-14}) \alpha_t -1551.86 \,.
\end{equation}
The utility function \ref{estimate_utility_function} is concave, as demonstrated in \autoref{fig_Utility_Function}.
\begin{figure}[H]
\centering
\centering{\includegraphics[width=0.45\linewidth]{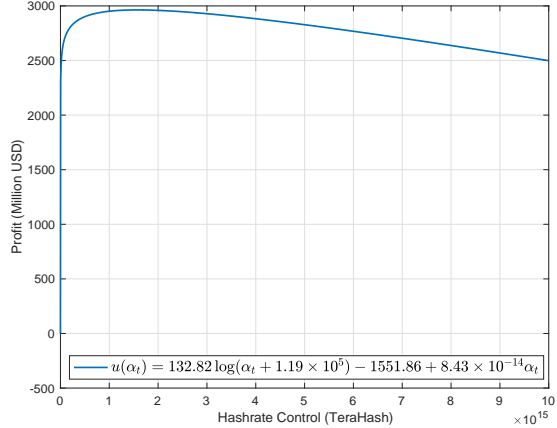}}
\caption{Utility function as described in \autoref{estimate_utility_function}.}
\label{fig_Utility_Function}
\end{figure}

Using Definition \ref{outcome}, the optimal control hashrate can be written as 
\[
\hat \alpha_t =
\arg\max \left\{ \theta_1 \log (\alpha_t + \theta_2) + \theta_3   - c \alpha_t 
+(rx - c \alpha_t ) \frac{\partial }{\partial x} v\left(x, b \right) 
+ \frac{\lambda_{t}}{h_t}\alpha_t \Big(v\left(x+k_t b, b\right)- v\left(x, b\right) \Big) 
\right\} \,.
\]
Taking the derivative with respect to optimal control renders a result that is equal to 0.
\[
\frac{\theta_1 }{\hat \alpha_t + \theta_2} - c  - c  \frac{\partial }{\partial x} v\left(x, b \right) 
+ \frac{\lambda_{t}}{h_t}  \Big(v\left(x+k_t b, b\right)- v\left(x, b\right) \Big) =0 \,.
\]
We can recover the optimal control: 
\begin{equation}\label{numerical_control}
{\hat \alpha_{t}}= \frac{\theta_1 }{ c  + c  \frac{\partial }{\partial x} v\left(x, b \right) 
+ \frac{\lambda_{t}}{h_t}  \left(v\left(x, b\right) - v\left(x+k_t b, b\right) \right) } - \theta_2 \,.
\end{equation}
A node with negative wealth does not possess resources to mine. We may apply a constraint that renders a node with zero optimal hashrate as inactive. Thus, for active nodes, we obtain
{
\footnotesize
\begin{equation}\label{active_miner}
O(t) = \left\{ \hat \alpha(x, b,t) >0  ~\Big |~ \hat \alpha(x, b,t) = \frac{{\theta_1} }{ c  + c  \frac{\partial }{\partial x} v\left(x, b \right) 
+ \frac{\lambda_{t}}{h_t}  \left(v\left(x, b\right) - v\left(x+k_t b, b\right) \right) } - \theta_2
\right\}\,.
\end{equation}
}To analyze the fixed point of the mapping described in Definition \ref{MEGF_def}, 
we tested our model with an increasing node growth model $M(t) = a x^b$ in the form of a power function, 
as shown in \autoref{fig_Number_of_Nodes_Model}.
In the fitted curve, the coefficients are $6.58 \times 10^{-3}  \in [4.29 \times 10^{-3}, 8.87 \times 10^{-3}]$ and $4.00 \in [3.94, 4.06]$ with $95\%$ confidence bounds.
We remark that modeling the number of nodes is a separate issue, which is beyond the scope of the current paper.

\subsection{Mean Hashrate}

A converged MFE consists of a long-run steady state as well as short-run dynamics, as shown in \autoref{fig_converged_mean_hashrate}. 
This figure represents the dynamic of computational power devoted to blockchain mining, which is measured by the number of trillion hashes computed per fortnight (TeraHash/Fortnight).

\begin{figure}[H]
\centering
{\includegraphics[width=0.5\linewidth]{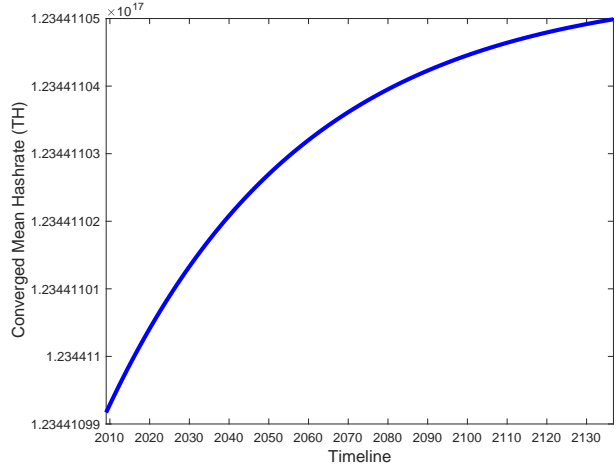}}
\caption{Converged mean hashrate over time, as described in definition \ref{MEGF_def}.
The discounted hashrate has increased from $1.23441099 \times 10^{17}$ to $1.23441105 \times 10^{17}$ . The increment is approximately $5.82 \times 10^{9}$ TeraHash. 
}
\label{fig_converged_mean_hashrate}
\end{figure}

The mean hashrate is non-decreasing over time, as shown in the blue line. This trend may be caused by the influence of market forces as new nodes enter the market to compete for rewards. As the number of nodes increases, the total hashrate will increase. The expected reward arrival rate and the amount of reward are adjusted in real time by the total hashrate. Changes in this hashrate may cause mining to be more or less profitable, and nodes will react by investing computational power accordingly as they pursue rewards.

As we mentioned in Game  \autoref{general_structure} \autoref{discounted_hashrate}, 
this hashrate represents the hashrate discounted by the rate of
the technological progress, which is equivalent to how much older generations of hardware can compute for the same cost. An increase in this discounted mean hashrate over time makes the underlying blockchain more secure, even though technological advancement in computational hardware is improving at an exponential rate.

The initial mean hashrate is designated in the PoW protocol, as given in \autoref{initial_hash}, 
requiring a large amount of the hashrate at the initial time period to build a secure blockchain. In reality, the mean hashrate of the Bitcoin blockchain did not have such large amounts available in the early days. However, it was an innovative product, and people did not think that it would be valuable enough to be attacked. By the time everyone thought it was, the strength of the Bitcoin network was already sufficiently resilient to attacks  \citep{antonopoulos2017mastering}.

\subsection{Density Evolution}

Population behavior in the MFE can be described through the evolution of density distribution. For any given initial probability density function $m_0(x_0,b_0)$,
the stationary distribution of active nodes at the terminal time can be found by solving time-independent PDEs 
\ref{system_pde},
as shown in \autoref{fig_initial_distibution}.

\begin{figure}[H]
\centering
\begin{subfigure}[t]{0.62\textwidth}
\subcaption{
}
\label{fig_steady_distibution}
\centering{\includegraphics[width=\linewidth]{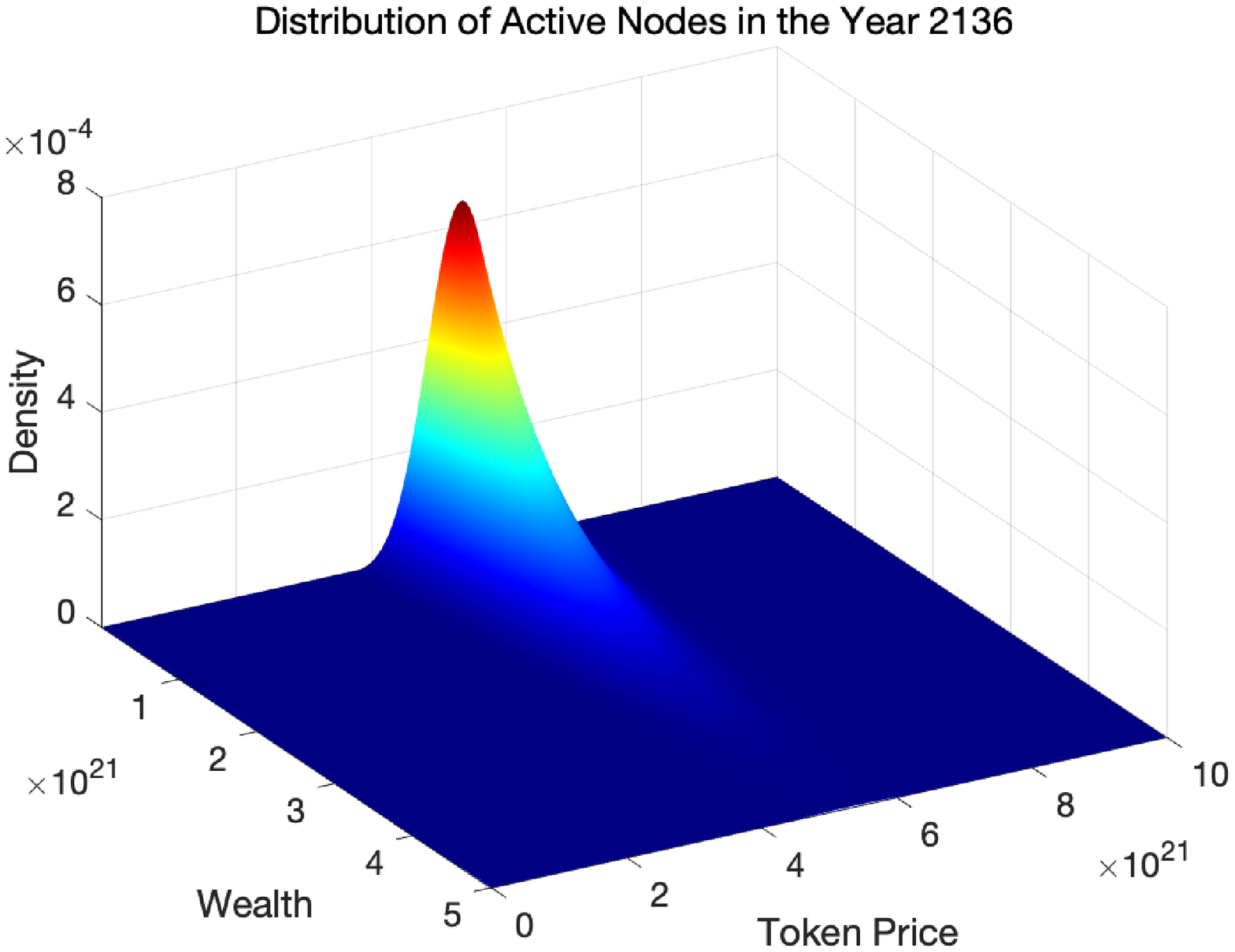}}
\end{subfigure}

\begin{subfigure}[t]{0.39\textwidth}
\subcaption{
}
\label{fig_steady_distibution_1}
\centering{\includegraphics[width=\linewidth]{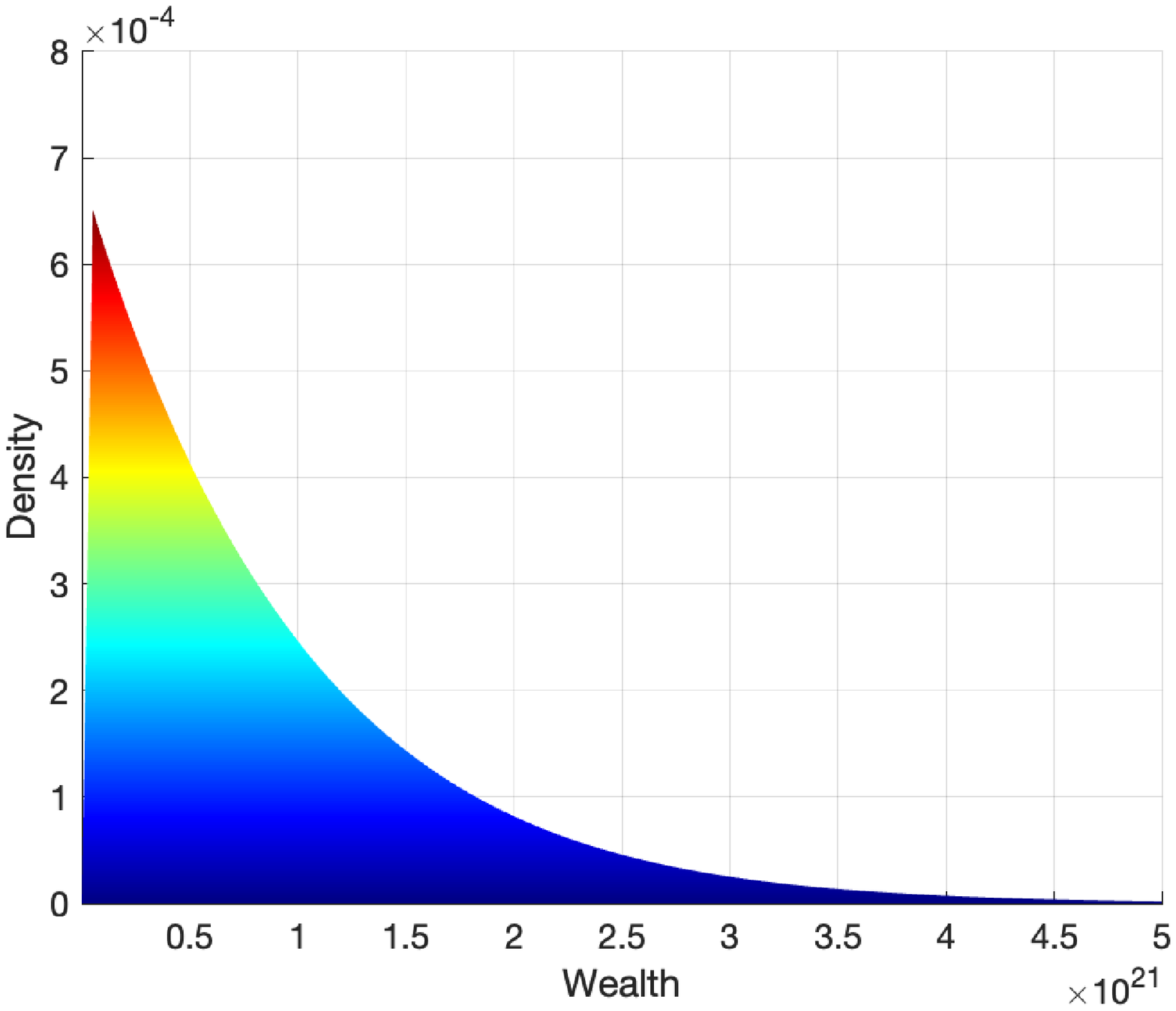}}
\end{subfigure}
\quad
\begin{subfigure}[t]{0.39\textwidth}
\subcaption{
}
\label{fig_steady_distibution_2}
\centering{\includegraphics[width=\linewidth]{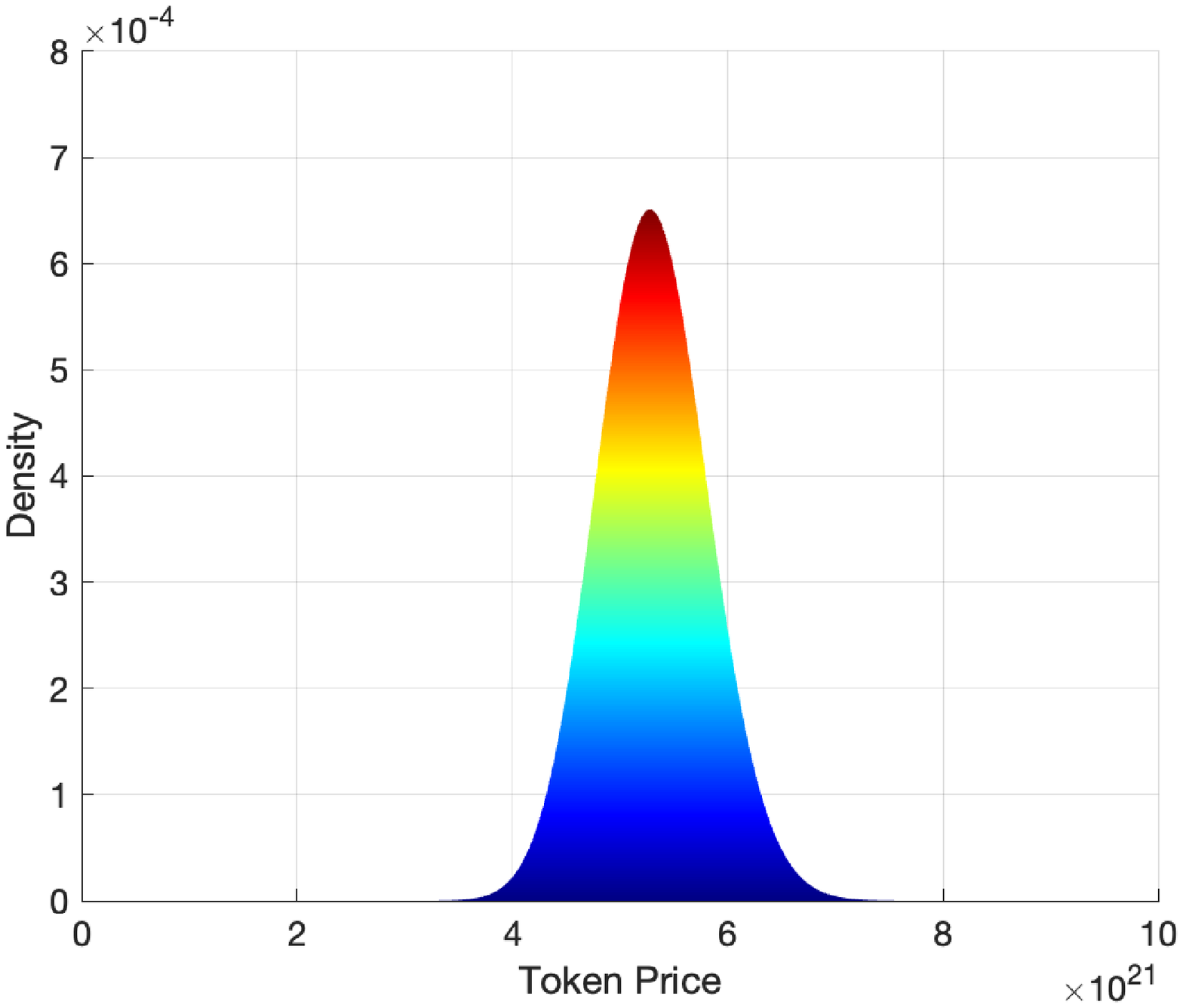}}
\end{subfigure}
\caption{
(\subref{fig_steady_distibution}) The stationary wealth-token price distribution of active nodes at the terminal time. 
(\subref{fig_steady_distibution_1}) and 
(\subref{fig_steady_distibution_2}) demonstrated 
the side views of (\subref{fig_steady_distibution}).
}
\label{fig_initial_distibution}
\end{figure}

The two-dimensional initial density is assumed to be accumulated at $b_0 = 0$ in the token price dimension and has an exponential distribution in the wealth dimension
\begin{equation}
m_0 (x_0, b_0) = \left(\frac{1}{\Delta x}\right) e^{-\frac{x}{\Delta x}} \,.
\end{equation}
Each node optimizes its total utility by solving a dynamic programming problem and recovering the optimal control from the value function. As every node uses its optimal control, a better recovery for the evolution of the density distribution is obtained. The collection of the resulting density distributions evolves with the optimal control strategy of each node over time, which is demonstrated in \autoref{evolution_density}.

\begin{figure}[H]
\centering
\begin{subfigure}[t]{0.43\textwidth}
\subcaption{
}
\label{Distribution_a}
\centering
{\includegraphics[width=\linewidth]{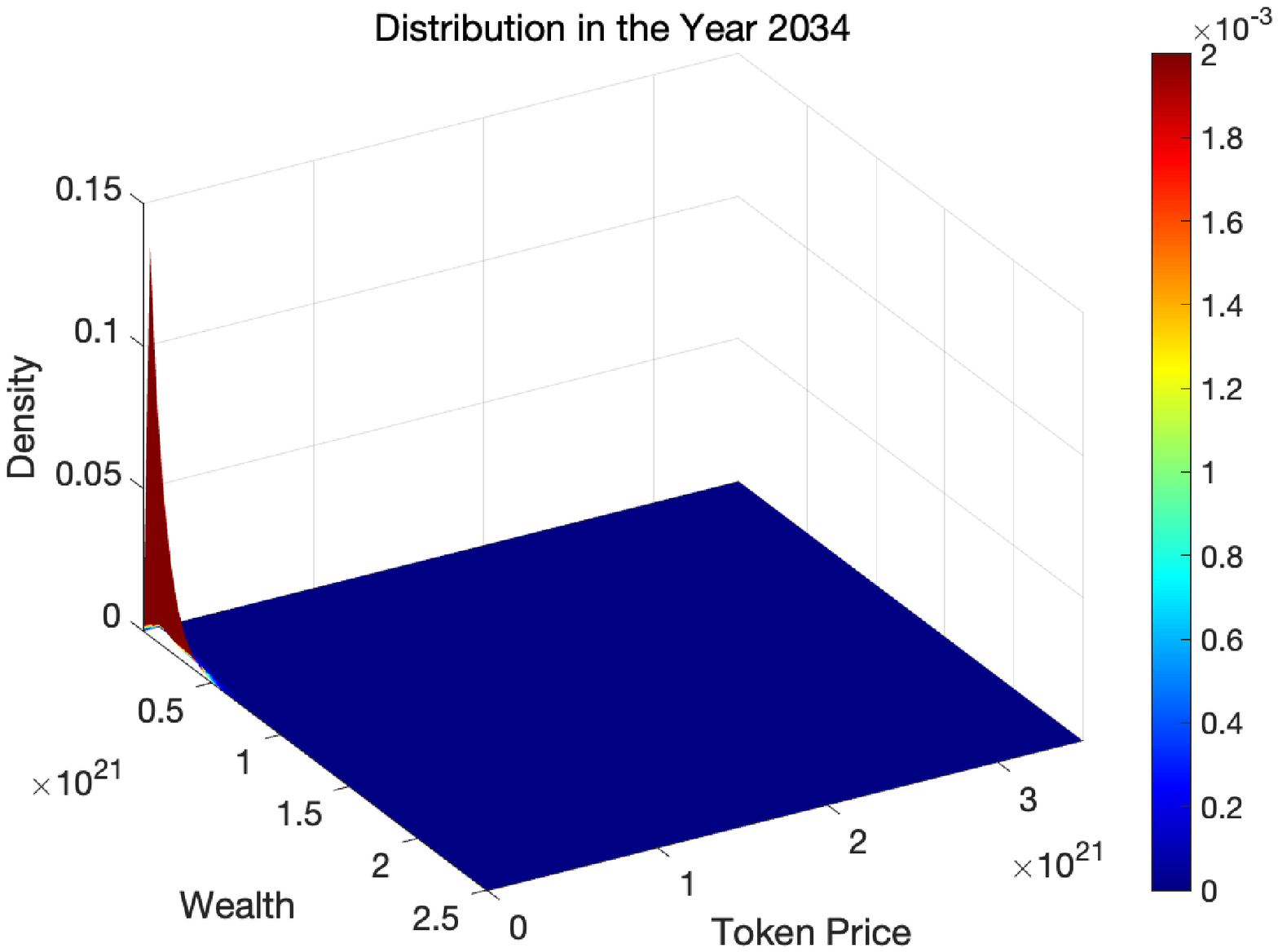}}
\end{subfigure}
\quad
\centering
\begin{subfigure}[t]{0.43\textwidth}
\subcaption{
}
\label{Distribution_b}
\centering
{\includegraphics[width=\linewidth]{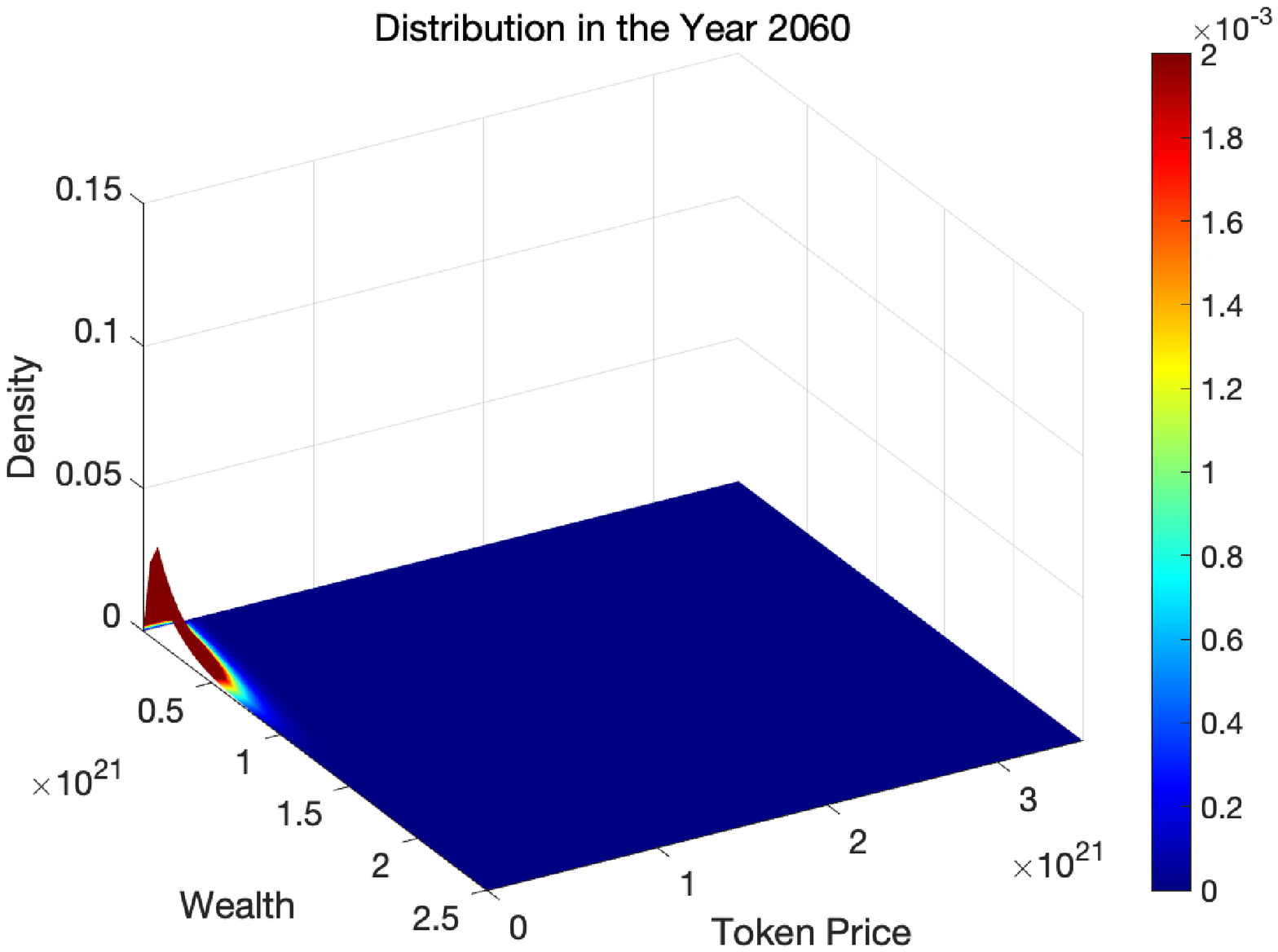}}
\end{subfigure}
\medskip
\centering
\begin{subfigure}[t]{0.43\textwidth}
\subcaption{
}
\label{Distribution_c}
\centering
{\includegraphics[width=\linewidth]{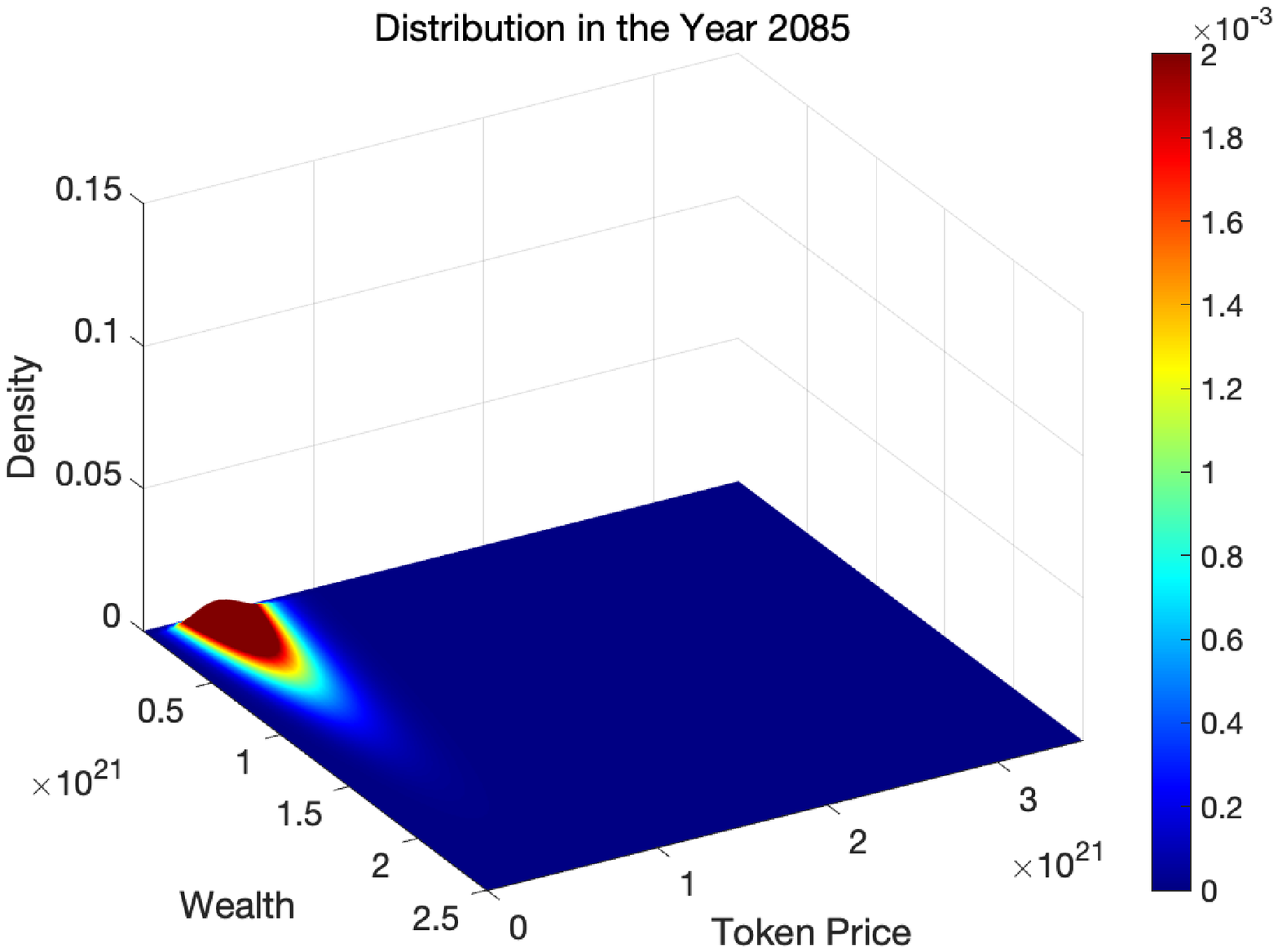}}
\end{subfigure}
\quad
\centering
\begin{subfigure}[t]{0.43\textwidth}
\subcaption{
}
\label{Distribution_d}
\centering
{\includegraphics[width=\linewidth]{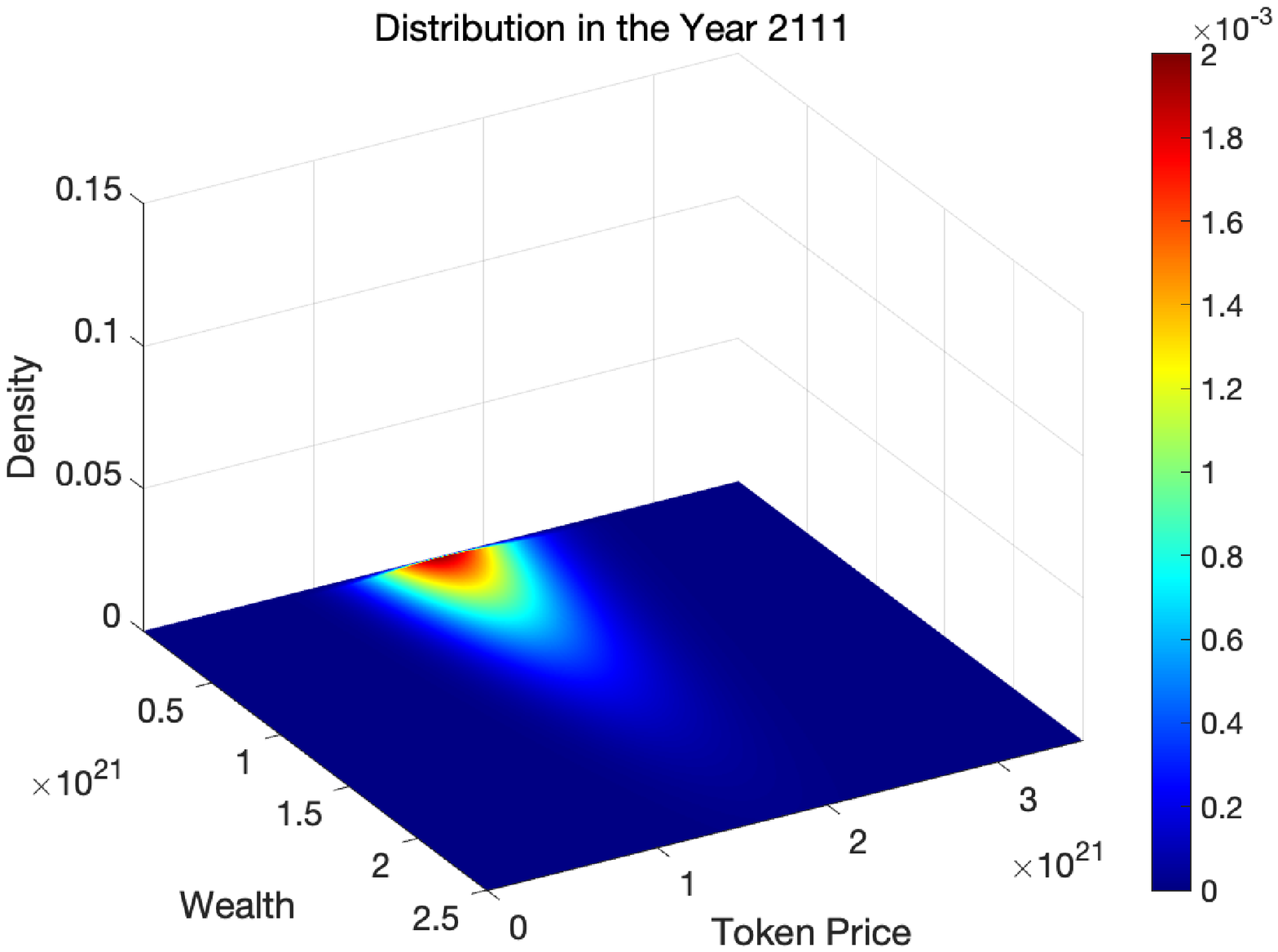}}
\end{subfigure}
\caption{
Evolution of the density function for active nodes over time. 
(\subref{Distribution_a}), (\subref{Distribution_b}), (\subref{Distribution_c}) and (\subref{Distribution_d}) show the evolution of the density function for active nodes in years 2040, 2072, 2104, and 2136, respectively.
}
\label{evolution_density}
\end{figure}

 In the token price dimension, distribution is driven by a mean reversion OU-process and is dependent on the current token price value. The total hashrate acts as an equilibrium level for the process.
Since the token price process is not directly controlled by individual nodes, 
we can reduce the dimensions of the state variables and focus on the evolution of the marginal density function with respect to the wealth state, as demonstrated in \autoref{Evolution of the Distribution}.

\begin{figure}[H]
\centering{\includegraphics[width=0.68\linewidth]{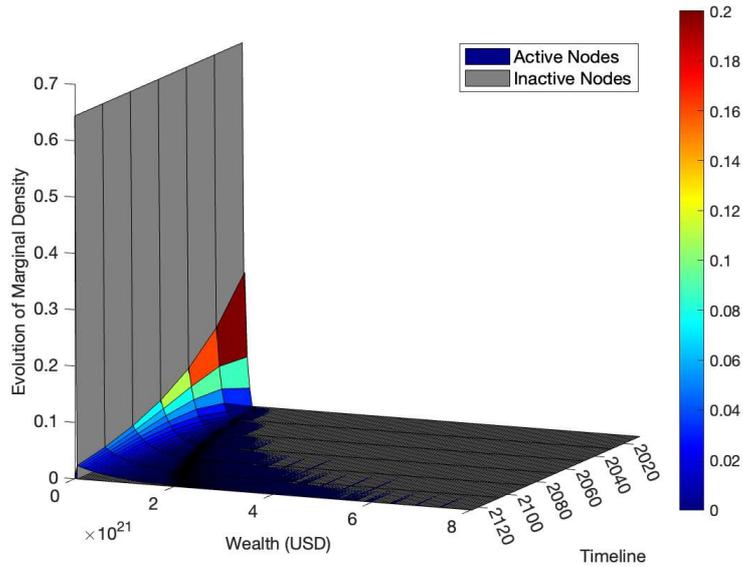}}
\caption{Marginal distribution with respect to the wealth state, as plotted on a timeline.
The majority of nodes lost their wealth and thus moved towards the left in the wealth marginal distribution.
The ergodic measure $m$ has a singularity on line $x=0$ in the wealth dimension, shaded in gray, which represents inactive nodes that have insufficient wealth and hence zero optimal controls.
While a smaller number of active nodes, who had more money originally, shaded in colors, moved to the right---that is, active nodes accumulated greater wealth over time.}
\label{Evolution of the Distribution}
\end{figure}

\autoref{Evolution of the Distribution} illustrates that
competition reduces the concentration of wealth as the number of nodes increases over time. This is because increases in competition are a source of risk to nodes and reduce their chances of gaining a reward, thereby diminishing the value of their wealth. Thus, the majority of the nodes becomes inactive and moves to the left of the wealth marginal distribution, forming a big spike.

Compared with the initial distribution, 
\autoref{Evolution of the Distribution} also indicates
the wealth of rich active nodes increases at a greater rate than poorer active nodes.
There are a few studies on user behavior and wealth accumulation in the Bitcoin network that suggest the rich have indeed become richer \citep{gupta2017gini}. 
Our numerical results agree with this phenomenon, suggesting that steps should be taken to curb such wealth accumulation in the network.

Wealth marginal distributions are further investigated in \autoref{wealth_distribution}.    
The active nodes of 
\autoref{Evolution of the Distribution} are plotted on a normalized log scale in \autoref{fig_Distribution_Evolution_2D_log},  
which shows that the wealth marginal distribution of the Bitcoin blockchain is skewed. 
We then compare our results to real world information. 
Given that the Bitcoin blockchain is a public data structure, 
we may estimate the distribution of wealth across all known Bitcoin addresses from publicly available blockchain information. This is plotted in blue stars \autoref{wealth_distribution_data}
 \footnote{The website such as https://bitinfocharts.com/top-100-richest-bitcoin-addresses.html that tracks the distribution of Bitcoin across all known Bitcoin addresses}. 
When making this comparison, it becomes clear that 
the shape of equilibrium distribution, as shown in the red line, 
resembles that of the real world data (plotted in blue stars). We cannot draw any stronger conclusions, because the data in \autoref{wealth_distribution_data} does not indicate which nodes those wallet addresses belong to; multiple wallet addresses may be controlled by the same node.

\begin{figure}[H]
\centering
\begin{subfigure}[t]{0.42\textwidth}
\subcaption{}
\label{fig_Distribution_Evolution_2D_log}
\centering{\includegraphics[width=0.9\linewidth]{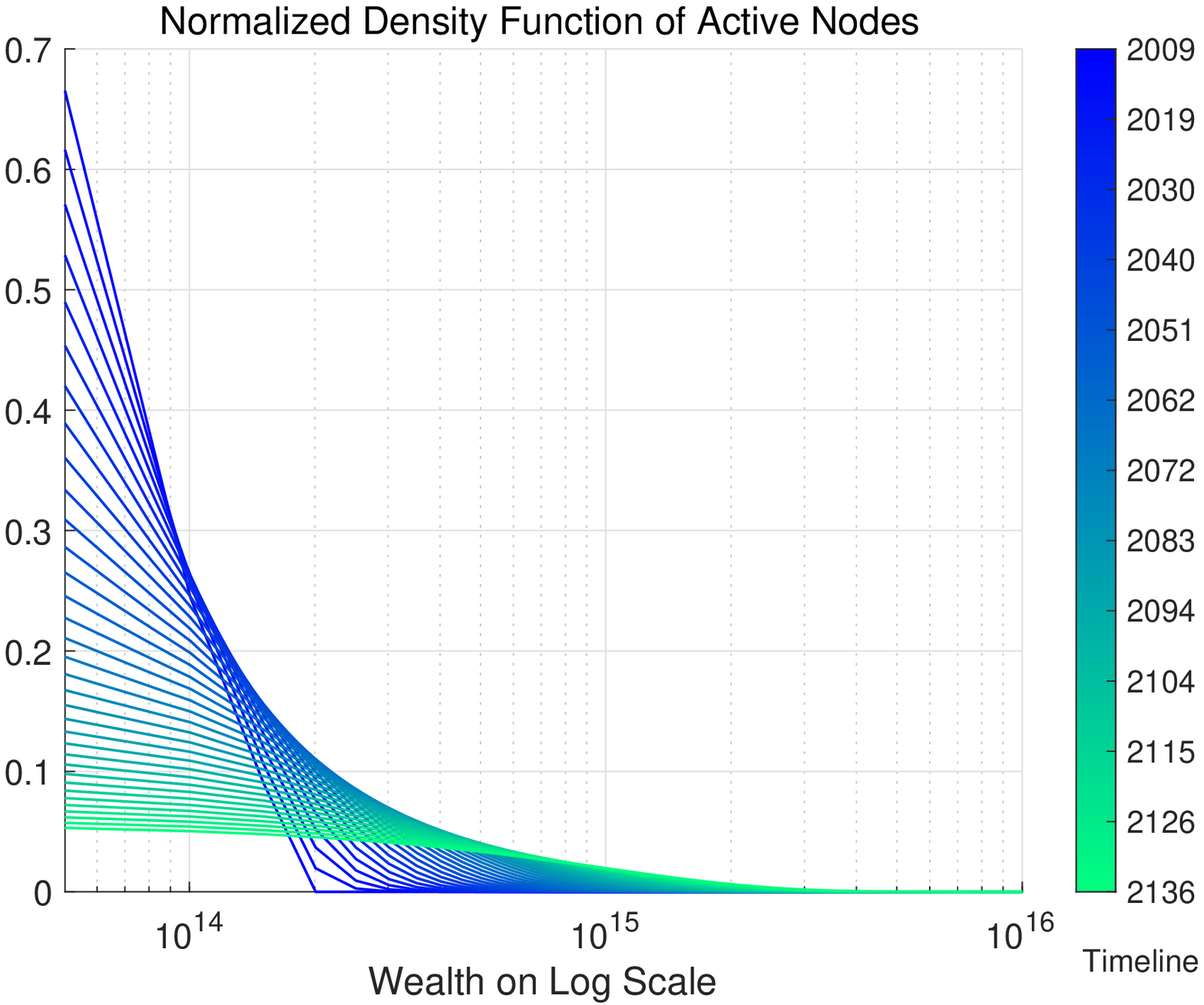}}
\end{subfigure}
\begin{subfigure}[t]{0.46\textwidth}
\subcaption{}
\label{wealth_distribution_data}
\vspace{-1mm}
\centering{\includegraphics[width=0.9\linewidth]{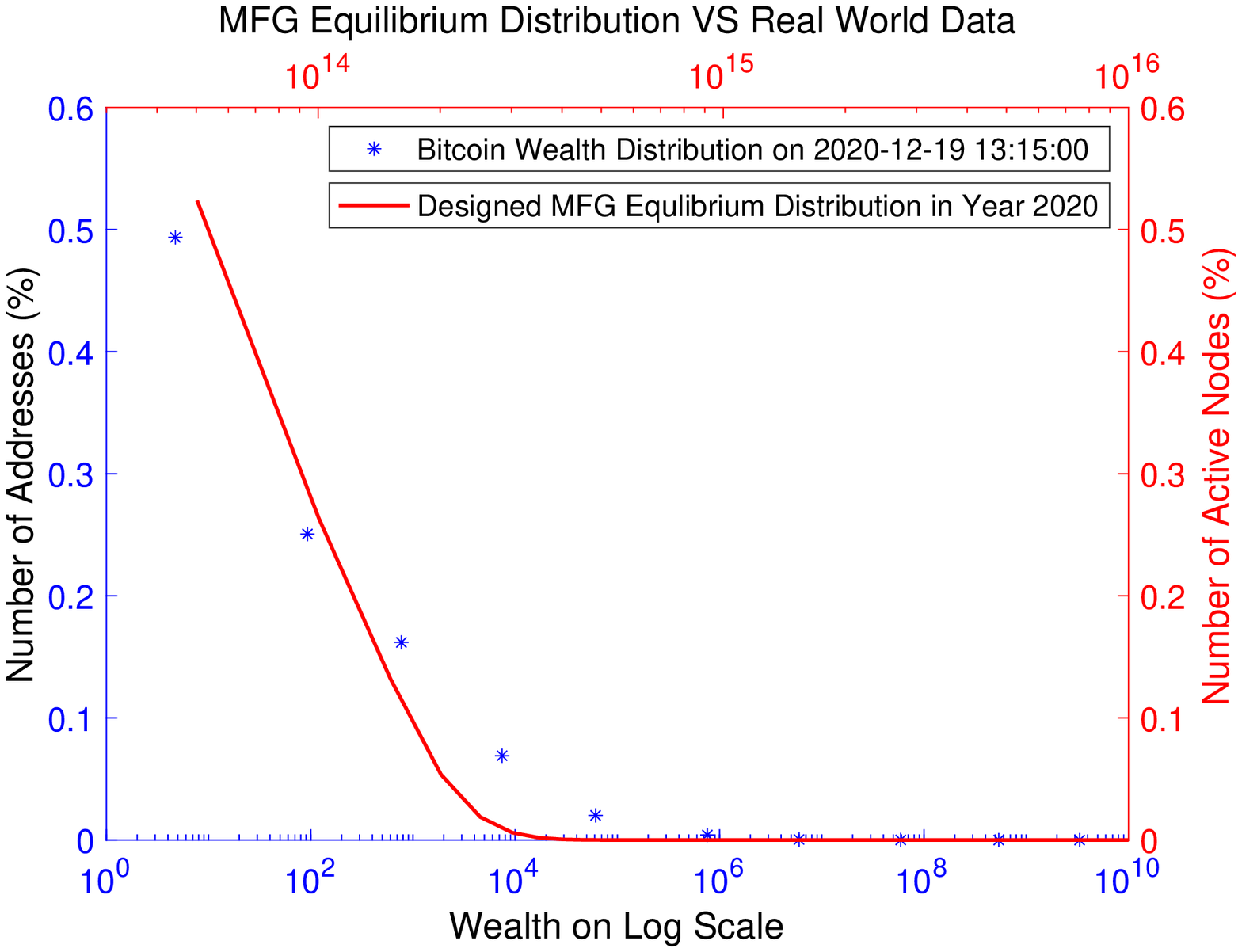}}
\end{subfigure}
\caption{
(\subref{fig_Distribution_Evolution_2D_log}) Evolutionary distribution on a log scale demonstrates that wealth distribution in the Bitcoin blockchain is skewed. (\subref{wealth_distribution_data}) Comparing MFG equilibrium distribution (plotted in the red line) to real world Bitcoin wealth distribution data on 19th Dec, 2020 (plotted in blue stars).
}
\label{wealth_distribution}
\end{figure}

\subsection{Mining Profitability}\label{result_profitbility}

The equilibrium optimal control discounted hashrate is approximately $1.58 \times 10^{15}$ TeraHashes for active nodes. To estimate the instantaneous expected utility, we may input this hashrate into \autoref{estimate_utility_function}. The resulting output is approximately 2963.21 USD. This means that mining is always profitable for active nodes who operate at the MFE of our model.

Further, we can estimate the number of active nodes at the equilibrium by using 
\begin{equation}\label{eq Number of Active Nodes}
\text{Number of Active Nodes} = 
M(t)\int_{X} \int_{B} m(t, x, b) \1_{\hat \alpha(t, x, b) \in O(t)}\diff x \diff b \,, 
\end{equation}
where the density function is illustrated in \autoref{Evolution of the Distribution} and the equilibrium optimal control is in feasible set \ref{active_miner}. The estimated number of active nodes over time is plotted in \autoref{fig_Number_of_Active_Nodes_Model}.  
\begin{figure}[H]
\centering{\includegraphics[width=0.5\linewidth]{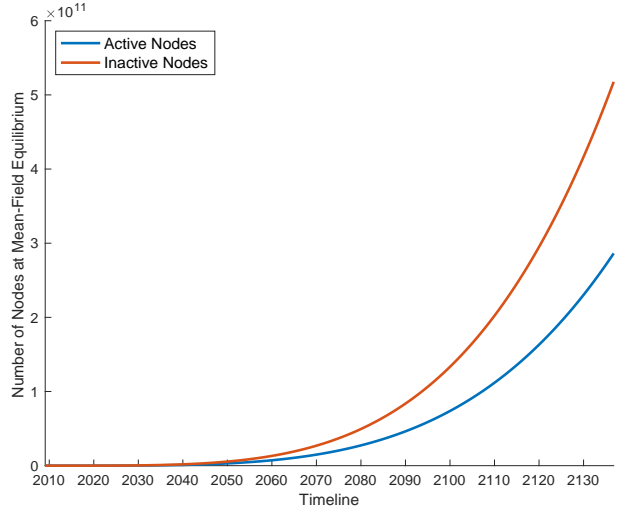}} 
\label{fig_Number_of_Active_Nodes_Model}
\caption{
Number of nodes at MFE, as plotted on a timeline.
The number of active nodes is obtained by using \autoref{eq Number of Active Nodes} (illustrated in blue); the number of inactive nodes is illustrated in red.
}
\label{fig_Number_of_Active_Nodes_Model}
\end{figure}
Even though the majority of nodes lose their wealth and become inactive over time, there is still an increase in active nodes. This represents the profitability of mining, and therefore the incentives for entering the mining market. Hence, the MFE contributes to the security of blockchain.

\subsection{Blockchain Security}\label{result_security}

Blockchain security is measured by the total hashrate that protects a record of valid transactions \citep{gervais2016security}.  
The most common attack is called the 51\% attack, a situation in which an attacker, 
who controls a majority of the hashrate, is probabilistically able to take over 50\% or more of the new block creation, 
or even double spend through rewriting part of the blockchain. 
\cite{nakamoto2008bitcoin} uses a double spend race analysis and demonstrates the exponential decay on the number of block confirmations for an attacker to quickly become computationally impractical.
\cite{grunspan2017satoshi} gvies
a closed-form formula for the probability of success of a double spend attack using a regularized incomplete beta function. 
A finer risk analysis on the probability of hashrate percentage to launch an attack
is provided in the result table of \cite{grunspan2018double}.
Consequently, we can use above results to
analyze 
such attacks by calculating 
the cost of launching an attack using
\[
\text{Percentage of the Total Hashrate Controlled by Attacker}  \times \text{Number of Active Nodes}   \times c \bar \alpha \,, \]
and 
estimating
the likelihood that nodes, like those in the wealth distribution of \autoref{Evolution of the Distribution}, can own majority of the total hashrate as it evolves over the time horizon. 
As expected, \autoref{fig_converged_mean_hashrate} depicts the increase in equilibrium mean hashrate. This makes it more expensive to control hashratess over time and therefore increases the security of the blockchain. The result is demonstrated in \autoref{fig_Blockchain_Security}.

\begin{figure}[H]
\centering{\includegraphics[width=0.5\linewidth]
{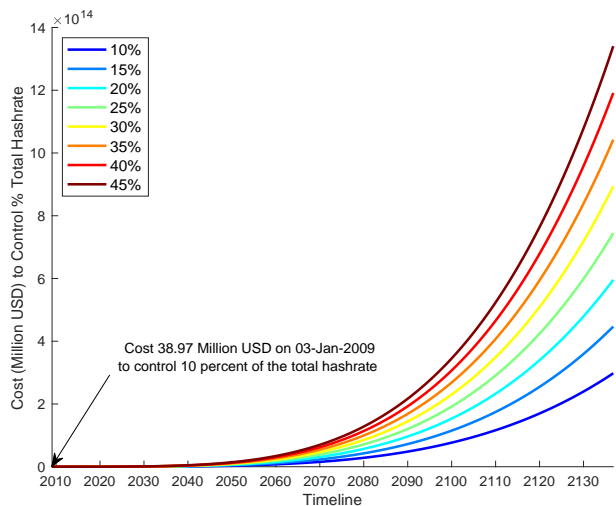}}

\caption{The cost for attackers
to control different percentage of the total hashrate over time.
A list of colors in legend is assigned to the cost to control 
10\%, 15\%, 20\%, 25\%, 30\%, 35\%, 40\% and 45\% of the total hashrate respectively.
The result shows it becomes more expensive to control hashrate over time, while
the corresponding success probability is discussed in \citep{nakamoto2008bitcoin}.
 }
\label{fig_Blockchain_Security}
\end{figure}

Overall, if mining is profitable at the equilibrium (as discussed in \Cref{result_profitbility}), and if the nodes in the network are unlikely to have the ability to launch attacks at the equilibrium (as discussed in \Cref{result_security}), then the blockchain can be regarded as secure. Consequently, if the blockchain is secure and ensures a decreasing inflation rate (Mechanisms \ref{reward_mechanism_1}--\ref{reward_mechanism_2}), then the blockchain’s consensus protocol is well designed. Generally, if people have faith in the security of the blockchain, then they will perceive the blockchain token as inherently valuable. Such a belief will catalyze the onset of a large number of decentralized nodes that invests resources (e.g., computing power) into maintaining the blockchain. This interdependence indirectly circulates and converts dollars into blockchain token value by creating a healthy mining ecosystem within the consensus protocol design \citep{narayanan2016bitcoin}.  In this way, 
 the consensus protocol acts as a mechanism for propagating the MFE over time.

\section{Conclusion}\label{section_5}

This paper started by formulating a consensus protocol design problem. 
We built a framework for a PoW protocol as an example to demonstrate a game theory approach to the problem. 
The traditional game theory approach makes MPE solutions complex; instead, we used a mean field game approximation and simplified the stochastic differential game model. The problem can be broken down into two coupled PDEs: where an individual node’s optimal control paths are analyzed through solving a Hamilton-Jacobi-Bellman equation; and where the evolution of the joint distributions of wealth and token price are characterized by a Fokker-Planck equation. 
We developed efficient numerical methods to compute both steady states at an infinite-time horizon and with time-dependent evolutionary distributions, and used real-time optimal control for each individual node. 
Our results demonstrate that the existence of the MFE represents the computational power devoted to building blocks in the Bitcoin blockchain. 
This MFE is a reliable approximation of a rational node’s behavior, in the sense that when other nodes use the MFE optimal controls in a finite stochastic differential game, the node’s best response is also to use the same optimal control. 
The MFE thus governs the 
security of the underlying blockchain. 
Hence, the consensus protocol in the Bitcoin blockchain can be regarded as well posed.
In conclusion, we can view blockchain as a mechanism that operates in a decentralized setup and propagates the MFE over time. This can help us to gain a deep understanding of a blockchain’s potential and limitations. 
In future research, this method may be used to explore related problems, including, but not limited to, 
ones concerning design transaction fees or centralization issues in blockchain networks. 
Our framework can also be extended to more complex structures,
or higher state dimensions. 
Further studies may explore numerical methods, such as deep learning and stochastic optimization techniques, to solve these type of  mean field forward-backward stochastic differential equations.

\section*{Acknowledgements}

L. Zhang is supported by the National Key R\(\&\)D Program of China 2021YFF1200500 and the National Natural Science Foundation of China 12050002. 
 Z. Zhou is supported by the National Key R\(\&\)D Program of China, Project Number 2020YFA0712000, 2021YFA1001200  and NSFC grant Number 12031013, 12171013.
 Also, we want to thank Cheng Feng Shen and Ari Klinger for helpful discussions.

\clearpage

\endgroup

\bibliographystyle{agsm}
\bibliography{Reference}

\begin{appendices}

\section{Proof-of-Work Protocol}\label{Appendix_POW}

\n 
In the Bitcoin blockchain, the game selects nodes in proportion to the amount of work done by the computing power devoted to building blocks, which is called the Proof-of-Work (PoW) protocol. This is intuitive because when all new transactions are broadcast to the network of nodes, the validity of those transactions as well as the new block needs to be verified by numerous confirmations. This requires work to be done\footnote{
Details of the work need to be done for maintaining one's own replica of the blockchain $ \mathcal{C}_i$ are listed in following steps \citep{nakamoto2008bitcoin}: 
1) verifying the validity of the new transactions and passing them to neighbours in the network;
2) building block $B(n)$ of new transactions according to consensus protocol $\mathfrak{C}$ \label{building_block},
and broadcasting the proposed block to the network;
3) verifying the proposed block according to consensus protocol $\mathfrak{C}$;
4) Nodes express their acceptance of the new block by working on creating the next block in the chain, using the hash of the accepted block as the previous hash for the new block they are mining.}. 
Defining a measurement for the amount of work done by computational power devoted to building blocks is equivalent to attempting a brute-force search of possible inputs to a cryptographic hash function whose output is trivial to verify but is computationally infeasible to invert. The Bitcoin blockchain chose the SHA256 Hash Function, something moderately hard that ensures the security of the blockchain  \citep{gilbert2003security}. \text{SHA-256} is a cryptographic hash function that maps a string of arbitrary text to a bit array of a fixed size
\[
\text{SHA-256}: \text{a string of text} \to \mathbb{Z}^+ \cup \{0\} \to \{0,1,\ldots,2^{256}-1\}
\]
such that for any given binary representation of the string $y$ in the codomain $\{ 0, 1 \}^{256}$ of $\text{SHA-256}$, it is computationally infeasible to find the element of pre-image $\text{SHA-256}^{-1}(y)$. That is, for mining each block, the only way to find a message that produces a given hash is to attempt a brute-force search of possible inputs to see if they produce a match. Hence, we can use above function to define a measurement.

\begin{mechanism}[Measurement Mechanism in the PoW Protocol]\label{PoW}

$\Psi$ in \autoref{measurement}  is chosen to be the PoW measurement $\alpha: \mathcal{C}  \rightarrow O$, which represents the computational power that a node invests in trying to solve the following puzzle:
\begin{enumerate}

\item Find a nonce such that the following inequality holds \footnote{This condition means when put 80-byte block header described in \autoref{block} and take 256-bit cryptographic hash function of this whole string twice,
then the hash output has to be less than a given 256 bit integer target.
In the rare case that the hash is less than the current target, 
the block is valid and is appended to the node's blockchain.
In the far more likely case that the hash is not less than the current target, 
the nonce is changed and the hash is recomputed.
Since this problem can only be solved by brute-force, more computing power will speed up the resolution. }
\begin{equation}\label{puzzle}
\text{SHA-256} \Big( \text{SHA-256} \big( \big\langle  \text{Merkle root} ~,~ \text{previous block's hash}  ~,~ \text{timestamp} ~,~ \text{nonce} ~ \big\rangle \big)  \Big) < \text{target}
\end{equation}
\item Search for a correct nonce in \autoref{puzzle} repeatly until someone finds and publishes the block \footnote{The nonce will be published as part of block}.
\item When a block is published to the network,
all nodes verify the validity of a new block and expands its local replica of blocktree
$ \mathcal{C}_i$ with the new valid block
\footnote{In the case of fork chain selection,  
nodes may resolve their inconsistency by setting their chain to be the longest chain}
\begin{equation}
{g} \big( \Psi^{-1}(\vO) \big) := \Big\{ C_j ~\Big| ~ j = \hspace{-0.1cm} \argmax_{i \in  \{1,2,\cdots,{M}\}} \valpha \Big\} \,. 
\end{equation} 
\end{enumerate}
\end{mechanism}

Each hash can be seen as an independent trial; in each trial, at time $t$, a nonce is selected and the block header is hashed, during which the node samples a value from a discrete uniform distribution with range
$
\left[0, 2^{256} - 1\right]
$. 
The difficulty $d(t)$, which is recorded in the header of every block represents how hard it is to mine a block compared with the original target.
The target at time $t$ is defined as
$
\frac{2^{224}}{d(t)}
$.
For each node, the probability of solving the block is the cumulative probability of selecting a value from $[0,\frac{2^{224}}{d(t)}]$ from all $2^{256}$ possible choices \citep{bowden2018block}, which is 
$
\frac{2^{224}}{d(t)}\div {2^{256}} = \frac{1}{2^{32}d(t)} 
$.

When ${M(t)}$ nodes are mining simultaneously, which can be modeled by a geometric distribution with a probability of success, there is 
\begin{equation}\label{eq:target_adj}
{1-\left(1-\frac{1}{2^{32}d(t)} \right)^{M(t)}}
\end{equation}
for every trial.

After choosing $\Psi$ as the PoW measurement, 
we need to design the payoff function $\pi$ by releasing rewards into the blockchain so that nodes to participate in the game; 
at the same time, we can also design the intensity of the block arrival process, as described in \autoref{def_poisson_intensity}, and the inflation rate of the token in this blockchain.

\begin{mechanism}[Reward Mechanism in PoW Protocol]\label{reward_mechanism}
In the Bitcoin blockchain, the reward is designed as follows:
\begin{enumerate}
\item A mining difficulty adjustment is made in order to maintain stability in the blockchain
\begin{itemize}\label{self_correcting}
\item mining difficulty dynamically adjusts every 2016 blocks \footnote{The previous 2016 blocks is found at an average block arrival rate of 1 blocks every 10 mins},
i.e,
the target in the ${s}^{\mathrm{th}}$ adjustment 
 is given by
\begin{equation}\label{eq:difficulty1}
d_{s} = \frac{1209600 d_{s-1}}{t_{2016s}-t_{2016(s-1)}} \quad \text{with}~~d_0 := 1
\end{equation}
\end{itemize}
\item The creator of a new block can choose a recipient address for block rewards in Bitcoin. These include 
\begin{itemize}\label{total_bitcoins}
\item a coin-creation transaction which starts at 50 and set to halve continually every 210,000 blocks until it reaches 0
\item transaction fees, which are purely voluntary (much like a tip).
\end{itemize}
\end{enumerate}
\end{mechanism}

\n For each time segment of 2,016 blocks denoted $s$, 
the probability of solving the block $\times$ total hashes over the time segment $\left(t_{2016(s-1)},t_{2016s}\right]$ = $2016$ successful blocks, 
applying \autoref{eq:target_adj} we get
\begin{equation}\label{ch2:eqn:1}
\left({1-\left(1-\frac{1}{2^{32} d_s} \right)^{M(t)}} \right) H_s = 2016
 \end{equation}
 where the total hashes over time segment $(t_{2016(s-1)}, t_{2016s}]$ is defined to be
 \[
 H_s := {\int_{t_{2016(s-1)}}^{t_{2016s}} h(t) \diff t} 
 \quad
 \mbox{where}
\quad
h(t) := \sum_{i \in  \{1,2,\cdots,{M(t)}\}} \alpha_i (t)  \,.
\]
 Rearranging \autoref{ch2:eqn:1}, we get
\begin{equation}\label{ch2:eqn:D}
d(t)_{t \in (t_{2016(s-1)}, t_{2016s}]} = d_s =  \frac{1}{2^{32} \left( 1- \sqrt[M(t)]{1- \frac{2016}{H_s}} \right)}
 \end{equation}
From Mechanism \autoref{reward_mechanism}  \autoref{self_correcting},
rearranging \autoref{eq:difficulty1} gives
\begin{equation}\label{eq:difficulty2}
 \frac{2016} {t_{2016s}-t_{2016(s-1)}} = \frac{ d_s}{600 d_{s-1}} 
\end{equation}
Substituting \autoref{ch2:eqn:D} into \autoref{eq:difficulty2},
a block arrival process over time interval $(t_{2016(s-1)}, t_{2016s}]$ can be viewed as a non-homogenous Poisson process $\left\{ N(t) \right\}_{t\geq 0}$ 
with an expected block arrival rate 
\begin{equation}
\left({{{\lambda}}}_t \right)_{t \in (t_{2016(s-1)}, t_{2016s}]} = \frac{2016}{t_{2016s}-t_{2016(s-1)}}
=  \frac{1}{600} \left( \frac{{  1- \sqrt[M(t)]{1- \frac{2016}{{H_{s-1}}} } }}{1- \sqrt[M(t)]{1- \frac{2016}{H_s}} }\right)
\end{equation}
Note that $d_0 := 1$ and \autoref{ch2:eqn:D} gives the initial condition
\[
H_0 := \frac{2016}{1-\left(1-\frac{1}{2^{32}} \right)^{M(t)}}
\]
Mechanism \autoref{reward_mechanism} \autoref{total_bitcoins} is designed to control the inflation.
Define $L := \{1, . . . , 32 \}$ time intervals, such that 
each time interval $l \in L$ includes 210000 blocks.
Thus the number of block reward at time $t$ can be written as
\begin{eqnarray}
k(t) 
= 50 \left(\frac{1}{2}\right)^{l} + \bar k
\end{eqnarray}
where
\[
l = \floor{\frac{2016 N_t}{210000}}
\]
and $\bar k \in [0, 1]$ is a small transaction fees paid in units of tokens.
Hence the number of cumulated tokens in circulation is 
\begin{equation}
{K(t)} = \left(210000 \times 50 \right) \left({\frac {1-(\frac{1}{2})^{l}}{1-\frac{1}{2}}}\right)+ \left(2016-\mod{(210000 l,2016)}\right)\left(50\left(\frac{1}{2}\right)^{l}\right)
\end{equation}

\end{appendices}

\end{document}